\documentclass[a4paper,reqno,10pt]{amsart}
\makeatletter
\newcommand*{\rom}[1]{\expandafter\@slowromancap\romannumeral #1@}
\makeatother
\usepackage{epsf,graphicx,epsfig}
\usepackage{amscd}
\usepackage{amsmath,latexsym,amssymb,amsthm}
\usepackage{fdsymbol}
\usepackage[nospace,noadjust]{cite}
\usepackage{textcomp}
\usepackage{setspace,cite}
\usepackage{mathrsfs,amsmath}
\usepackage{chemarrow}
\usepackage{lscape,fancyhdr,fancybox}
\usepackage{stmaryrd}
\usepackage[all,cmtip]{xy}
\usepackage{tikz-cd}
\usepackage{cancel}
\usepackage[usestackEOL]{stackengine}
\usetikzlibrary{shapes,arrows,decorations.markings}
\usepackage{graphicx}

\usepackage{color}
\usepackage{bbm, dsfont}
\usepackage{xcolor}
\usepackage{url}
\usepackage{enumerate}
\usepackage[mathscr]{euscript}

  \theoremstyle{plain}
\swapnumbers
    \newtheorem{thm}{Theorem}[section]
    \newtheorem{proposition}[thm]{Proposition}

    \newtheorem{subsec}[thm]{}
\theoremstyle{definition}
    \newtheorem{definition}[thm]{Definition}
        \newtheorem{remark}[thm]{Remark}
    \newtheorem{exam}[thm]{Example}

\theoremstyle{remark}

\newcommand{\dcb}[2]{\{\!\!\{#1, #2\}\!\!\}}

\title{}
\author{}
\date{}
\usepackage{amssymb}

\usepackage{hyperref}
\hypersetup{
colorlinks,
citecolor=blue,
filecolor=black,
linkcolor=blue,
urlcolor=black
}

\usepackage[left=25mm,right=25mm,top=25mm,bottom=25mm,paper=a4paper]{geometry}

\begin{document}

\title{Nijenhuis deformations of Poisson algebras and $F$-manifold algebras}

\author{Anusuiya Baishya}
\address{Department of Mathematics,
Indian Institute of Technology, Kharagpur 721302, West Bengal, India.}
\email{anusuiyabaishya530@gmail.com}

\author{Apurba Das}
\address{Department of Mathematics,
Indian Institute of Technology, Kharagpur 721302, West Bengal, India.}
\email{apurbadas348@gmail.com, apurbadas348@maths.iitkgp.ac.in}

\maketitle


\begin{abstract}
   The notion of pre-Poisson algebras was introduced by Aguiar in his study of zinbiel algebras and pre-Lie algebras. In this paper, we first introduce NS-Poisson algebras as a generalization of both Poisson algebras and pre-Poisson algebras. An NS-Poisson algebra has an associated sub-adjacent Poisson algebra. We show that a Nijenhuis operator and a twisted Rota-Baxter operator on a Poisson algebra deforms the structure into an NS-Poisson algebra. The semi-classical limit of an NS-algebra deformation and a suitable filtration of an NS-algebra produce NS-Poisson algebras. On the other hand, $F$-manifold algebras were introduced by Dotsenko as the underlying algebraic structure of $F$-manifolds. We also introduce NS-$F$-manifold algebras as a simultaneous generalization of NS-Poisson algebras, $F$-manifold algebras and pre-$F$-manifold algebras. In the end, we show that Nijenhuis deformations of $F$-manifold algebras and the semi-classical limits of NS-pre-Lie algebra deformations have NS-$F$-manifold algebra structures.
\end{abstract}

\medskip

\medskip
\begin{center}

    

    \end{center}

    \tableofcontents

\section{Introduction}

\subsection{Poisson algebras and {\em F}-manifold algebras} Poisson manifolds appear naturally in Hamiltonian mechanics and are also central objects in the study of deformation quantization \cite{kont}. Poisson algebras are the underlying algebraic structure of Poisson manifolds. Explicitly, a {\em Poisson algebra} is a commutative associative algebra $(A, \cdot)$ equipped with a Lie bracket $\{ ~, ~ \} : A \times A \rightarrow A$ that satisfies the Leibniz rule:
\begin{align}\label{leib-rule}
    \{ x, y \cdot z \} = \{ x, y \} \cdot z +  y \cdot \{ x, z \}, \text{ for } x, y, z \in A.
\end{align}
The algebra of observables in symplectic geometry, the space of smooth functions on the dual of a Lie algebra, vertex operator algebras and the semi-classical limit of any associative formal deformations of commutative associative algebras are examples of Poisson algebras. Representations of Poisson algebras are also extensively studied in the literature \cite{caressa}, \cite{skryabin}. Recall that a {\em representation} of a Poisson algebra $(A, \cdot , \{~,~\})$ is a triple $(V, \mu , \rho)$ in which $(V, \mu)$ is a representation of the commutative associative algebra $(A, \cdot)$ and $(V, \rho)$ is a representation of the Lie algebra $(A, \{~,~\})$ satisfying additionally 
\begin{align}
    \rho_x \circ \mu_y - \mu_y \circ \rho_x - \mu_{ \{x,y\} }=0, \label{rep-pois1}\\
    \mu_x \circ \rho_y+ \mu_y \circ \rho_x - \rho_{ x\cdot y}=0, \label{rep-pois2}
\end{align}
for all $x, y \in A$.
See also \cite{kosmann} for more details about Poisson algebras.

On the other hand, the notion of an {\em $F$-manifold} (also called {\em weak Frobenius manifold}) was introduced by Hertling and Manin \cite{hertling-manin} as a generalization of Frobenius manifold that appeared in $2$-dimensional topological field theories. An {\em $F$-manifold} is a pair $(M, \circ)$ consisting of a smooth supermanifold $M$ and a smooth bilinear commutative associative product $\circ$ on the tangent sheaf $TM$ such that
\begin{align*}
    P_{X_1 \circ X_2} (X_3, X_4) = X_1 \circ P_{X_2} (X_3, X_4) + (-1)^{|X_1| | X_2 |} X_2 \circ P_{X_1} (X_3, X_4),
\end{align*}
where $P_{X_1} (X_2, X_3) = [X_1, X_2 \circ X_3] - [X_1, X_2] \circ X_3 - (-1)^{|X_1| |X_2|} X_2 \circ [X_1, X_3]$ measures to what extent the product $\circ$ and the Lie bracket $[~,~]$ of vector fields fail to satisfy the Leibniz rule of a Poisson algebra (see also \cite{merkulov}). In \cite{dotsenko}, Dotsenko first considered the notion of an {\em $F$-manifold algebra} as the underlying algebraic structure of the $F$-manifolds. He further showed that the graded object of the filtration of the operad {\em PreLie} is the operad {\em FMan} encoding $F$-manifold algebras. It has been observed by Liu, Sheng and Bai \cite{liu-sheng-bai} that $F$-manifold algebras are the semi-classical limits of pre-Lie formal deformations of commutative pre-Lie algebras. Thus, it follows that $F$-manifold algebras are related to pre-Lie algebras in the same way Poisson algebras are related to associative algebras.

\subsection{Rota-Baxter deformations and pre-algebras}
The notion of {\em Rota-Baxter operators} originated in the probability study of Baxter \cite{baxter}, promoted by the combinatorial study of Rota \cite{rota}, and found important applications in mathematics and mathematical physics. In \cite{aguiar}, Aguiar first observed that a Rota-Baxter operator on an associative algebra deforms the structure into a {\em dendriform algebra}, a structure that splits the usual associativity. More generally, a Rota-Baxter operator on a class of binary quadratic nonsymmetric operads gives the black square product of dendriform algebra with these operads \cite{fard-guo,vall}. For arbitrary binary quadratic operads, the above result was generalized in \cite{bai}, hence setting the general operadic definition for the notion of splitting of algebraic structures. More generally, given a binary quadratic operad $\mathcal{P}$, they consider an operad $\mathrm{BSu}(\mathcal{P})$ the bisuccessor of $\mathcal{P}$. It has been observed that the operad $\mathrm{BSu}(\mathcal{P})$ is equivalent to taking the Manin black product of the operad $PreLie$ and the operad $\mathcal{P}$. The algebras over the operad $\mathrm{BSu}(\mathcal{P})$ are called pre-$\mathcal{P}$-algebras. When $\mathcal{P} = Ass, Com, Lie$, the corresponding pre-$\mathcal{P}$-algebras are respectively dendriform algebras, zinbiel algebras and pre-Lie algebras. When $\mathcal{P} = PreLie$, the corresponding pre-$\mathcal{P}$-algebra is called an {\em L-dendriform algebra} introduced in \cite{bai-l}. In \cite{aguiar}, Aguiar introduced the notion of a {\em pre-Poisson algebra} by combining a zinbiel algebra and a pre-Lie algebra structure on the same vector space. Pre-Poisson algebras are nothing but pre-$\mathcal{P}$-algebras for $\mathcal{P} = Poiss$, the operad for Poisson algebras. He also showed that the semi-classical limit of dendriform deformation of a zinbiel algebra inherits a pre-Poisson algebra structure.


\subsection{Nijenhuis deformations, twisted Rota-Baxter deformations and NS-algebras}
The notion of {\em NS-algebras} (also called NS-associative algebras) was introduced by Leroux \cite{leroux} in the study of Nijenhuis operators on associative algebras. Roughly, an NS-algebra is a vector space equipped with three bilinear operations $\prec, \succ$ and $\curlyvee$ satisfying certain axioms that imply that the new operation $\ast = \prec + \succ + \curlyvee$ is associative. NS-algebras generalize dendriform algebras and split the associativity. In \cite{uchino}, Uchino also considered NS-algebras and showed that NS-algebras can be obtained from twisted Rota-Baxter operators (in particular, Reynolds operators) on associative algebras. Recently, the author in \cite{das} introduced the notion of an {\em NS-Lie algebra} while studying Nijenhuis operators and twisted Rota-Baxter operators on Lie algebras. NS-Lie algebras simultaneously generalize Lie algebras and pre-Lie algebras. Some more details about NS-algebras, NS-Lie algebras and other generalizations can be found in \cite{das,das2,guo-lei,ospel}. The above discussions lead to the question about the effect of Nijenhuis operators and twisted Rota-Baxter operators on a binary quadratic operad $\mathcal{P}$. This is equivalent to understanding NS-$\mathcal{P}$-algebras for any binary quadratic operad $\mathcal{P}$. However, unlike pre-$\mathcal{P}$-algebras, there has been no well-known description of NS-$\mathcal{P}$-algebras in terms of the operad $\mathcal{P}$. Thus, to understand NS-$\mathcal{P}$-algebras for some given $\mathcal{P}$, one needs to explicitly write down the effect of Nijenhuis operators on $\mathcal{P}$-algebras. For instance, the authors in \cite{guo-zhang} considered {\em NS-pre-Lie algebras} in their study of Nijenhuis operators and twisted Rota-Baxter operators on pre-Lie algebras.

\subsection{NS-Poisson algebras and NS-{\em F}-manifold algebras}
Our aim in this paper is to study the effect of Nijenhuis operators on Poisson algebras and $F$-manifold algebras. We first introduce the notion of an {\em NS-Poisson algebra} by combining NS-commutative algebra and NS-Lie algebra on the same vector space (cf. Definition \ref{nspois}). Our notion of NS-Poisson algebras generalizes both Poisson algebras and pre-Poisson algebras. We show that an NS-Poisson algebra gives rise to a Poisson algebra structure on the underlying vector space, called the {\em subadjacent} Poisson algebra (cf. Theorem \ref{prop-subadj-p}). Thus, an NS-Poisson algebra can be realized as a splitting of a Poisson algebra. On the other hand, any Nijenhuis operator on a Poisson algebra deforms the structure into an NS-Poisson algebra (cf. Proposition \ref{prop-nij-nsp}). It follows that an NS-Poisson algebra is the underlying structure of Nijenhuis deformations of a Poisson algebra. Next, we consider twisted Rota-Baxter operators on Poisson algebras and show that they also induce NS-Poisson algebras (cf. Theorem \ref{thm-twisted-rota-ns}). Conversely, any NS-Poisson algebra arise from a twisted Rota-Baxter operator (cf. Proposition \ref{prop-last}). Since a Reynolds operator on a Poisson algebra can be regarded as a twisted Rota-Baxter operator, it also induces an NS-Poisson algebra structure (cf. Proposition \ref{prop-rey-nsp}). Next, we consider NS-algebra deformations of an NS-commutative algebra and show that NS-Poisson algebras are the corresponding semi-classical limit (cf. Theorem \ref{thm-defor}). Finally, we show that a suitable filtration of an NS-algebra induces an NS-Poisson algebra structure on the associated graded vector space (cf. Theorem \ref{thm-filt}).

By viewing $F$-manifold algebras as a generalization of Poisson algebras, one could also study the effect of Nijenhuis operators on an $F$-manifold algebra. In another part of the paper, we introduce the notion of an {\em NS-$F$-manifold algebra} as a simultaneous generalization of $F$-manifold algebras, pre-$F$-manifold algebras \cite{liu-sheng-bai} and NS-Poisson algebras. Any NS-$F$-manifold algebra naturally gives rise to an $F$-manifold algebra structure (cf. Theorem \ref{thm-subadj-nsf}). We show that NS-$F$-manifold algebras are the underlying structure of Nijenhuis deformations and Reynolds deformations of an $F$-manifold algebra (cf. Propositions \ref{prop-nij-nsf} and \ref{prop-rey-nsf}). Among other results, we observe that NS-$F$-manifold algebras are related to NS-pre-Lie algebras in the same way $F$-manifold algebras are related to pre-Lie algebras. Explicitly, we consider NS-pre-Lie deformations of an NS-commutative algebra (which can be viewed as an NS-algebra and hence as an NS-pre-Lie algebra) and show that NS-$F$-manifold algebras are the corresponding semi-classical limit (cf. Theorem \ref{thm-nsf-semi-classical}). In particular, pre-$F$-manifold algebras are the semi-classical limit of L-dendriform deformations of zinbiel algebras. These results together with some previously known results can be summarized in the following diagram:

\medskip

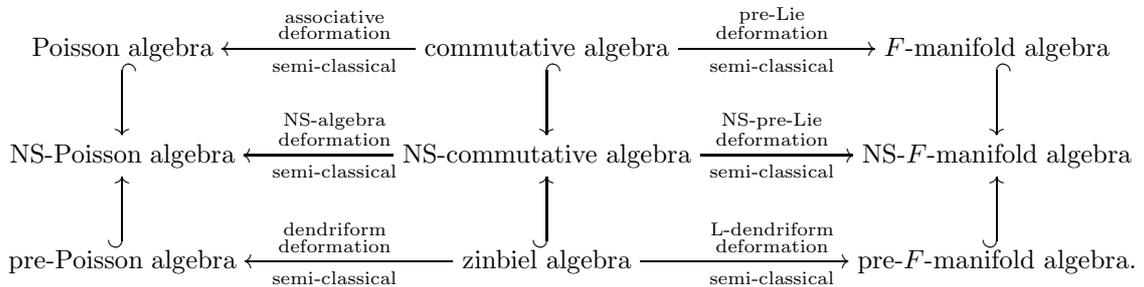
\[
\xymatrix{
\text{Poisson algebra} \ar@{^{(}->}[d] & &  \text{commutative algebra}   \ar@{^{(}->}[d] \ar[rr]^{\substack{\text{pre-Lie}\\ \text{deformation}}}_{ \text{semi-classical}} \ar[ll]_{\substack{\text{associative}\\ \text{deformation}}}^{ \text{semi-classical}} & & F\text{-manifold algebra}  \ar@{^{(}->}[d] \\
\text{NS-Poisson algebra} & & \text{NS-commutative algebra} \ar[rr]^{\substack{\text{NS-pre-Lie}\\ \text{deformation}}}_{ \text{semi-classical}} \ar[ll]_{\substack{\text{NS-algebra}\\ \text{deformation}}}^{ \text{semi-classical}} & & \text{NS-}F\text{-manifold algebra}\\
\text{pre-Poisson algebra} \ar@{^{(}->}[u] & & \text{zinbiel algebra}  \ar@{^{(}->}[u] \ar[rr]^{\substack{\text{L-dendriform}\\ \text{deformation}}}_{ \text{semi-classical}} \ar[ll]_{\substack{\text{dendriform}\\ \text{deformation}}}^{ \text{semi-classical}} & & \text{pre-}F\text{-manifold algebra.}  \ar@{^{(}->}[u] \\
}
\]

\medskip

\medskip

In \cite{aguiar}, Aguiar also introduced the notion of a dual pre-Poisson algebra. It has been observed that an averaging operator on a Poisson algebra gives rise to a dual pre-Poisson algebra. Generalizing his results, the authors of \cite{liu-sheng-bai} introduced dual pre-$F$-manifold algebras. In both the above formulations, it is used that the operad {\em Poiss} of Poisson algebras is self-dual and the dual of pre-$\mathcal{P}$ is the operad encoding the structures obtained by applying averaging operators on $\mathcal{P}$-algebras. The last statement means that Rota-Baxter operators and averaging operators are dual to each other in an appropriate sense. In a subsequent paper, we will demonstrate Nijenhuis operators from operadic points of view and find their dual. Using this, we aim to formulate dual NS-Poisson algebras and dual NS-$F$-manifold algebras.


\subsection{Organization of the paper}
    The paper is organized as follows. In section \ref{section2}, we recall pre-Poisson algebras, NS-algebras, NS-Lie algebras and Nijenhuis operators along with some relevant results. In section \ref{section3}, we introduce NS-Poisson algebras and explore their relations with Poisson algebras, Nijenhuis operators, twisted Rota-Baxter operators and Reynolds operators. We also introduce and study some basic properties of NS-Gerstenhaber algebras as a graded version of NS-Poisson algebras. In section \ref{section4}, we show that the semi-classical limit of an NS-algebra deformation of a given NS-commutative algebra carries an NS-Poisson algebra structure. We also show that a suitable filtration of an NS-algebra gives rise to an NS-Poisson algebra structure on the associated graded space. We begin section \ref{section5} by introducing the notion of NS-$F$-manifold algebras and find their relations with $F$-manifold algebras, Nijenhuis operators and Reynolds operators. In the end, we prove that NS-pre-Lie deformations of an NS-commutative algebra gives rise to NS-$F$-manifold algebras as semi-classical limits.

    All vector spaces, (multi)linear maps and tensor products are over a field {\bf k} of characteristic $0$.

\section{Pre-algebras and NS-algebras} \label{section2}
In this section, we recall various pre-algebras (e.g. zinbiel algebras, pre-Lie algebras and pre-Poisson algebras), NS-algebras and NS-Lie algebras. 

    \begin{definition}
        A \textbf{(left) zinbiel algebra} is a pair $(A, \ast)$ of a vector space $A$ with a bilinear operation $\ast: A \times A \rightarrow A$ satisfying
            \begin{equation}
                x \ast (y \ast z) = (x \ast y + y \ast x ) \ast z,  \text{ for } x, y, z \in A.
            \end{equation}
    \end{definition}

    \begin{definition}
        A \textbf{(left) pre-Lie algebra} is a pair $(A, \diamond)$ consisting of a vector space $A$ equipped with a bilinear operation $\diamond: A \times A \rightarrow A$ satisfying the following identity: 
            \begin{equation}
                x \diamond (y \diamond z) - (x \diamond y) \diamond z = y \diamond (x \diamond z) - (y \diamond x ) \diamond z, \text{ for } x, y, z \in A.
            \end{equation}
    \end{definition}

    \begin{definition}
        A \textbf{(left) pre-Poisson algebra} is a triple $(A, \ast, \diamond)$ in which $ (A, \ast)$ is a zinbiel algebra and $(A,\diamond )$ is a pre-Lie algebra satisfying the following compatibilities: for any $x, y, z \in A$,
            \begin{align}
                 (x \diamond y - y \diamond x) \ast z =~& x \diamond ( y \ast z )- y \ast (x \diamond z),\\
                (x \ast y + y \ast x) \diamond z =~& x \ast (y \diamond z ) + y \ast (x \diamond z).
            \end{align}
            \end{definition}

            Let $(A, \ast, \diamond)$ be a pre-Poisson algebra. Then it has been shown in \cite{aguiar} that $(A, \cdot, \{ ~, ~ \})$ becomes a Poisson algebra, where
            \begin{align*}
                x \cdot y := x \ast y + y \ast x \quad \text{ and } \quad \{ x, y \} := x \diamond y - y \diamond x, \text{ for } x, y \in A.
            \end{align*}
            This is called the sub-adjacent Poisson algebra of the pre-Poisson algebra $(A, \ast, \diamond)$.

            See \cite{aguiar} for some examples of pre-Poisson algebras. In particular, a Rota-Baxter operator on a Poisson algebra gives rise to a pre-Poisson algebra. Next, we consider NS-algebras introduced by Leroux \cite{leroux} and NS-Lie algebras introduced in \cite{das}.

    \begin{definition} \label{nsalg}
         An {\bf NS-associative algebra} (or simply an {\bf NS-algebra}) is a quadruple $(A, \prec , \succ, \curlyvee)$ consisting of a vector space $A$ with three bilinear operations $\prec , \succ, \curlyvee : A \times A \rightarrow A$ subject to satisfy the following identities:
            \begin{align}
                (x \prec y) \prec z =&~ x \prec ( y \odot z), \label{nsalg1}\\
                (x \succ y ) \prec z =&~ x \succ ( y \prec z),\\
                (x \odot y) \succ z =&~ x \succ( y \succ z ),\\
                (x \curlyvee y ) \prec z + (x \odot y) \curlyvee z =&~ x \succ (y \curlyvee z) + x \curlyvee (y \odot z), \label{nsalg4}
            \end{align}
        for $x, y, z \in A$, where we used the notation $x \odot y = x \prec y + x \succ y + x \curlyvee y$.
    \end{definition}
    
    \begin{remark}
        Let $(A, \prec , \succ , \curlyvee)$ be an NS-algebra. Then it follows from the identities (\ref{nsalg1})-(\ref{nsalg4}) that the bilinear operation $\odot$ is associative.
    \end{remark}

    For an NS-algebra $(A, \prec, \succ, \curlyvee )$, if the symmetries $x \prec y = y \succ x$ and $x \curlyvee y = y \curlyvee x$ hold for all $x, y \in A$, then we call it as an NS-commutative algebra. Thus, an {\bf NS-commutative algebra} can be described by a triple $(A, \ast, \curlyvee)$, where $\ast , \curlyvee : A \times A \rightarrow A$ are bilinear operations in which $\curlyvee$ is commutative (i.e. $x \curlyvee y = y \curlyvee x$, for all $x, y \in A$) satisfying the following identities:
        \begin{align}
            x \ast ( y \ast z) =~& (x \odot y) \ast z,\\
            x \ast ( y \curlyvee z) + x \curlyvee ( y \odot z) =~& y \ast ( x \curlyvee z) + y \curlyvee ( x \odot z),
        \end{align}
        for all $x, y , z \in A$. Here $x\odot y= x\ast y+y\ast x+x\curlyvee y$.


     Any NS-commutative algebra $(A, \ast, \curlyvee)$ with trivial $\curlyvee$ is simply a zinbiel algebra. Thus, NS-commutative algebras are a generalization of zinbiel algebras.

    \begin{remark} \label{repcom}
        Let $(A, \ast, \curlyvee)$ be an NS-commutative algebra. Then $(A, \odot)$ is a commutative associative algebra, called the subadjacent algebra. Moreover, the map $\mu : A \rightarrow \mathrm{End}(A), x \mapsto \mu_x $ given by
            \begin{align} \label{mu}
                \mu_x(y) :=x\ast y, \text{~~~for~} x, y \in A,
            \end{align}
       defines a representation of the commutative associative algebra $(A, \odot)$ on the vector space $A$.
    \end{remark}

    Let $(A, \cdot)$ be a commutative associative algebra. A linear map $N:A \rightarrow A$ is said to be a {\bf Nijenhuis operator} on $A$ if it satisfies
        \begin{align} \label{ncom}
            N(x) \cdot N(y) = N( N(x) \cdot y + x \cdot N(y) - N(x\cdot y) ), \text{ for all } x, y \in A.
        \end{align}
 \begin{proposition} \label{nc}
        Let $(A, \cdot)$ be a commutative associative algebra and $N: A\rightarrow A$ be a Nijenhuis operator on it. We define bilinear operations $\ast , \curlyvee : A \times A \rightarrow A$ by
            \begin{align*}
                x \ast y = N(x)\cdot y \text{~~and~~} x\curlyvee y= -N(x\cdot y), \text{ for } x, y \in A.
            \end{align*}
        Then $(A, \ast, \curlyvee )$ is an NS-commutative algebra. Moreover, the corresponding subadjacent commutative associative algebra is given by $(A, \cdot_N)$, where $x \cdot_N y:= N(x) \cdot y + x \cdot N(y) - N(x\cdot y)$, for $x, y \in A$. Additionally, we have $N(x \cdot_N y)= N(x)\cdot N(y),$ for $x, y \in A$.
        \end{proposition}
        

    \begin{definition} \label{nslie}
        An {\bf NS-Lie algebra} is a triple $(A, \diamond, \smallblackdiamond )$ consisting of a vector space $A$ with two bilinear operations $\diamond, \smallblackdiamond : A \times A \rightarrow A$ in which the operation $\smallblackdiamond$ is skew-symmetric and the following identities hold:
            \begin{align}
                & x \diamond ( y \diamond z) - (x \diamond y) \diamond z - y \diamond( x \diamond z) + ( y \diamond x ) \diamond z = (x \smallblackdiamond y) \diamond z, \label{nsl1}\\
                 & x \smallblackdiamond \dcb{y}{z} + y \smallblackdiamond \dcb{z}{x} + z \smallblackdiamond \dcb{x}{y} + x \diamond (y \smallblackdiamond z) + y \diamond (z \smallblackdiamond x) + z \diamond (x \smallblackdiamond y) = 0,  \label{nsl2}
            \end{align}
           for $x, y, z \in A$. Here we used the notation $\dcb{x}{y}= x \diamond y - y \diamond x + x \smallblackdiamond y$.
    \end{definition}

        Any NS-Lie algebra $(A, \diamond, \smallblackdiamond)$ with trivial $\smallblackdiamond$ is simply a pre-Lie algebra. On the other hand, an NS-Lie algebra $(A, \diamond, \smallblackdiamond)$ with trivial $\diamond$ is nothing but a Lie algebra. Thus, NS-Lie algebras are a generalization of both pre-Lie algebras and Lie algebras.

    \begin{remark}  \label{replie}
        Let $(A, \diamond, \smallblackdiamond)$ be an NS-Lie algebra. Then it has been observed in \cite{das} that  
        $(A, \dcb{~}{~})$ is a Lie algebra. This is called the subadjacent Lie algebra of $(A, \diamond, \smallblackdiamond)$ and it is denoted by $A^c$.
        Moreover, the map $\rho : A^c \rightarrow \mathrm{End}(A), x \mapsto \rho_x$ given by
            \begin{equation} \label{rho}
                \rho_x (y)= x \diamond y,  \text{ ~for~} x,y \in A,
            \end{equation}
        defines a representation of the subadjacent Lie algebra $A^c$ on the vector space $A$.
    \end{remark}

    Let $(A, \{ ~, ~\})$ be a Lie algebra. A linear map $N:A \rightarrow A$ is called a {\bf Nijenhuis operator} on the Lie algebra $(A,\{~,~\})$ if it satisfies
    \begin{align} \label{nlie}
        \{N(x) , N(y)\} = N( \{N(x),y\} + \{x,N(y)\} - N\{x, y\} ), \text{ for all } x, y \in A.
    \end{align}
     

     The following result has been proved in \cite{das}.
     
    \begin{proposition} \label{nl}
       Let $(A, \{~,~\})$ be a Lie algebra and $N:A\rightarrow A$ be a Nijenhuis operator on it. We define bilinear operations $\diamond , \smallblackdiamond : A \times A \rightarrow A$ by
        \begin{align*}
            x \diamond y = \{N(x),y\} \text{~~and~~} x\smallblackdiamond y= -N\{x,y\}, \text{ for } x, y \in A.
        \end{align*}
        Then $(A, \diamond, \smallblackdiamond)$ is an NS-Lie algebra. Moreover, the corresponding subadjacent Lie algebra is given by $(A, \{ ~, ~\}_N)$, where $\{ x, y \}_N := \{N(x),y\} + \{x, N(y)\} - N\{x, y\}$, for $x, y \in A$. Additionally, we have $N (\{x,y\}_N) =\{N(x) , N(y)\}$, for $x, y \in A$.
    \end{proposition}    
    
\section{NS-Poisson algebras} \label{section3}
In this section, we introduce NS-Poisson algebras as a generalization of both pre-Poisson algebras and Poisson algebras. We show that an NS-Poisson algebra naturally induces a Poisson algebra structure. Further, NS-Poisson algebras can be seen as the underlying structure of Nijenhuis operators and twisted Rota-Baxter operators on Poisson algebras. In fact, we observe that any NS-Poisson algebra arises from a twisted Rota-Baxter operator. We end this section by considering NS-Gerstenhaber algebras as the graded version of NS-Poisson algebras.

    \begin{definition} \label{nspois}
        An \textbf{NS-Poisson algebra} is a quintuple $(A, \ast, \curlyvee, \diamond, \smallblackdiamond)$ consisting of a vector space $A$ with four bilinear operations $\ast$, $\curlyvee$, $\diamond$, $\smallblackdiamond$ $:A \times A \rightarrow A$ such that
        $( A, \ast, \curlyvee)$ is an NS-commutative algebra, $( A, \diamond, \smallblackdiamond )$ is an NS-Lie algebra and the following compatibilities hold:
            \begin{align*}
                \dcb{x}{y} \ast z  =&~ x \diamond ( y \ast z) - y \ast (x \diamond z ), \tag{NSP1} \label{nsp1}\\
                ( x \odot y) \diamond z =&~ x \ast (y \diamond z) + y \ast (x \diamond z ),  \tag{NSP2} \label{nsp2}\\
                 x \smallblackdiamond ( y \odot z ) + x \diamond ( y \curlyvee z) =&~ \dcb{x}{y} \curlyvee z + z \ast (x \smallblackdiamond y ) + y \curlyvee \dcb{x}{z} + y \ast (x \smallblackdiamond z), \tag{NSP3} \label{nsp3}
            \end{align*}
       for $x, y, z \in A$. Here $x \odot y= x \ast y + y \ast x + x \curlyvee y$ and $ \dcb{x}{y} = x \diamond y - y \diamond x + x \smallblackdiamond y $.
    \end{definition}

    \begin{remark}
        Let $(A, \ast, \curlyvee, \diamond, \smallblackdiamond)$ be an NS-Poisson algebra in which $\curlyvee$ and $\smallblackdiamond$ are trivial. Then $(A, \ast, \diamond)$ turns out to be a pre-Poisson algebra.  On the other hand, if $(A, \ast, \curlyvee, \diamond, \smallblackdiamond)$ is an NS-Poisson algebra in which $\ast$ and $\diamond$ are trivial then $(A, \curlyvee, \smallblackdiamond = \{ ~, ~ \})$ becomes a Poisson algebra. Thus, NS-Poisson algebras can be regarded as a common generalization of both pre-Poisson algebras and Poisson algebras.
    \end{remark}

    \begin{thm}\label{prop-subadj-p}
        Let $(A, \ast , \curlyvee , \diamond , \smallblackdiamond)$ be an NS-Poisson algebra.
        \begin{itemize}
            \item[(i)] Then $(A, \odot , \dcb{~}{~})$ is a Poisson algebra, where $\odot$ and $\dcb{~}{~}$ are defined in Definition \ref{nspois}. (This is called the sub-adjacent Poisson algebra of $(A, \ast,\curlyvee,\diamond, \smallblackdiamond)$ and is denoted by $A^c$.)
            \item[(ii)] Moreover, the triple $(A,\mu , \rho)$ is a representation of the sub-adjacent Poisson algebra $A^c$, where $\mu$ and $\rho$ are given by (\ref{mu}) and (\ref{rho}), respectively.
        \end{itemize}          
    \end{thm}

    \begin{proof}
      (i)  Since $(A, \ast, \curlyvee)$ is an NS-commutative algebra, it follows that $(A, \odot)$ is a commutative associative algebra. On the other hand,
         $(A, \diamond , \smallblackdiamond)$ is an NS-Lie algebra implies that $(A, \dcb{~}{~})$ is a Lie algebra. To show that $(A, \odot, \dcb{~}{~})$ is a Poisson algebra, we need to verify the Leibniz rule. For any $x, y, z \in A$, we have
        \begin{align*}
            &\dcb{x}{y}\odot z + \dcb{x}{z}\odot y\\
            &=\dcb{x}{y}\ast z + z \ast \dcb{x}{y} + \dcb{x}{y} \curlyvee z + \dcb{x}{z} \ast y + y \ast \dcb{x}{z} + \dcb{x}{z} \curlyvee y\\
            &=\dcb{x}{y} \ast z + z \ast (x \diamond y) - z \ast (y \diamond x) + z \ast (x \smallblackdiamond y) + \dcb{x}{y} \curlyvee z + \dcb{x}{z} \ast y \\
            &\quad + y \ast (x \diamond z) - y \ast (z \diamond x) + y \ast (x \smallblackdiamond z) + \dcb{x}{z} \curlyvee y\\
            &=\big( \dcb{x}{y} \ast z + y \ast (x \diamond z)\big) + \big(z \ast (x \diamond y) + \dcb{x}{z} \ast y\big) - \big(z\ast (y \diamond x) + y \ast (z \diamond x)\big)\\
            & \quad + \big(z \ast (x \smallblackdiamond y) +\dcb{x}{y} \curlyvee z+ y \ast (x \smallblackdiamond z) + \dcb{x}{z} \curlyvee y \big )\\
            &=x \diamond (y \ast z) + x \diamond (z \ast y)- (y \odot z) \diamond x + x \smallblackdiamond (y \odot z) + x \diamond (y \curlyvee z)\\
            &=x \diamond (y \odot z) - (y \odot z) \diamond x + x \smallblackdiamond (y \odot z) \\
            &=  \dcb{x}{y \odot z}
        \end{align*}
        which verifies the Leibniz rule. Hence $(A, \odot, \dcb{~}{~})$ is a Poisson algebra.

        (ii) We have already seen in Remark \ref{repcom} that $(A, \mu)$ defines a representation of $(A, \odot)$. Similarly, in Remark \ref{replie} we have seen that $(A, \rho)$ defines a representation of the Lie algebra $(A, \dcb{~}{~})$. Thus, to show that $(A, \mu, \rho)$ is a representation of the sub-adjacent Poisson algebra $A^c = (A, \ast , \curlyvee , \diamond , \smallblackdiamond)$, we need to verify the identities (\ref{rep-pois1}) and (\ref{rep-pois2}).

        Since $(A, \ast, \curlyvee, \diamond, \smallblackdiamond)$ is an NS-Poisson algebra, it follows from (\ref{nsp1}) that
        \begin{align*}
           \mu_{\dcb{x}{y}} (z) = (\rho_x \circ \mu_y -\mu_y \circ \rho_x )(z), \text{ for } x, y, z \in A.
        \end{align*}
Hence the identity (\ref{rep-pois1}) follows. On the other hand, it follows from (\ref{nsp2}) that 
\begin{align*}
    \rho_{x\odot y}(z) = (\mu_x \circ \rho_y + \mu_y \circ \rho_x )(z), \text{ for } x, y, z \in A.
\end{align*}
This verifies the identity (\ref{rep-pois2}). Hence complete the proof.
  \end{proof}

         The above construction of the sub-adjacent Poisson algebra of an NS-Poisson algebra generalizes the similar construction for pre-Poisson algebras.
         
         In the following, we consider Nijenhuis operators on a Poisson algebra and show that a Nijenhuis operator induces an NS-Poisson algebra structure.
 Let $(A, \cdot, \{ ~, ~ \})$ be a Poisson algebra. Recall that a {\bf Nijenhuis operator} on the Poisson algebra $(A, \cdot, \{~,~\})$ is a linear map $N:A \rightarrow A$ that is a Nijenhuis operator on both the commutative associative algebra $(A, \cdot)$ and the Lie algebra $(A, \{~,~\})$. Thus, $N$ is a Nijenhuis operator if both the identities (\ref{ncom}) and (\ref{nlie}) hold.
    
    \begin{proposition}\label{prop-nij-nsp}
        Let $ ( A, \cdot, \{ ~, ~\})$ be a Poisson algebra and $N: A \rightarrow A$ be a Nijenhuis operator on it. Define four bilinear operations $\ast , \curlyvee , \diamond , \smallblackdiamond : A \times A \rightarrow A$ by 
        $$ x \ast y = N(x) \cdot y , \quad  x \curlyvee y = - N(x \cdot y), \quad x \diamond y = \{N(x), y\} \quad \text{ and } \quad  x \smallblackdiamond y = - N\{x,y\}, \text{ for } x, y \in A.$$
        Then $(A, \ast , \curlyvee , \diamond , \smallblackdiamond)$ is an NS-Poisson algebra.
    \end{proposition}
    
    \begin{proof}
        Since $N: A \rightarrow A$ is a Nijenhuis operator on the commutative associative algebra $(A, \cdot)$, it follows from Proposition \ref{nc} that $(A, \ast, \curlyvee)$ is an NS-commutative algebra. On the other hand, since $N$ is also a Nijenhuis operator on the Lie algebra $(A, \{~,~\})$, it follows from Proposition \ref{nl} that $(A, \diamond, \smallblackdiamond)$ is an NS-Lie algebra. Thus, to show that $(A, \ast, \curlyvee, \diamond,\smallblackdiamond)$ is an NS-Poisson algebra, we need to verify the identities (\ref{nsp1}), (\ref{nsp2}) and (\ref{nsp3}). For any $x, y, z \in A$, we first observe that
        \begin{align*}
            x \diamond (y \ast z) - y \ast (x \diamond z)
            &= \{ N(x), N(y)\cdot z \}- N(y) \cdot \{ N(x), z \} \\
            &= \{ N(x) , N(y) \} \cdot z\\
            &= N ( \{ x, y \}_N) \cdot z\\
            &= \dcb{x}{y} \ast z.
        \end{align*}
        This verifies the identity (\ref{nsp1}). Similarly, we have
        \begin{align*}
            x \ast (y \diamond z) + y \ast (x \diamond z)
            &= N(x) \cdot \{ N(y), z \} + N(y) \cdot \{N(x), z\}\\
            &= \{ N(x) \cdot N(y) , z\}\\
            &= \{N(x \cdot_N y),z\}\\
            &= (x \odot y) \diamond z,
        \end{align*}
        which verifies (\ref{nsp2}). Finally, we have
        \begin{align*}
            & x\smallblackdiamond (y \odot z) + x\diamond (y\curlyvee z)-\dcb{x}{y}\curlyvee z-z\ast (x \smallblackdiamond y)-y\curlyvee \dcb{x}{z}-y\ast (x\smallblackdiamond z)\\
            &= -N\{x,y\cdot_N z\}-\{N(x),N(y\cdot z)\}+N(\{x,y\}_N\cdot z)+N(z)\cdot N\{x,y\}+N(y\cdot \{x,z\}_N)\\
            & \quad + N(y)\cdot N\{x,z\}\\
            &= N \big( -\{x,y\cdot_N z\}-\{x,N(y\cdot z) \} -\{N(x),y\cdot z\}+N\{x, y \cdot z\} +\{x,y\}_N\cdot z+ z\cdot N\{x,y\}\\
            & \quad + N(z) \cdot \{x,y\}- N(z\cdot \{x,y\})+y\cdot \{x,z\}_N + y\cdot N\{x,z\} + N(y)\cdot \{x,z\}- N(y\cdot \{x,z\} )  \big)\\
            &= N \Big(  \big(-\{x,y\cdot_N z\}-\{x,N(y\cdot z)\}\big)-\{N(x),y\cdot z\}+N\big({\{x, y\cdot z\}}-{z\cdot \{x,y \}}- {y\cdot \{x,z\}}\big)\\
            & \quad +\big(\{x,y\}_N \cdot z + z\cdot N\{x,y\}\big) + N(z)\cdot \{x,y\} + \big(y\cdot \{x,z\}_N + y\cdot N\{x,z\}\big) +N(y)\cdot \{x,z\} \Big)\\
            &= N \big( -\{x,N(y)\cdot z\}-\{x, y \cdot N(z)\}- \{N(x), y\cdot z\}+ \{N(x),y\}\cdot z+ \{x,N(y)\}\cdot z+ N(z)\cdot\{x,y\}\\
            & \quad +y\cdot \{N(x),z\}+ y\cdot \{x,N(z)\}+N(y)\cdot \{x,z\} \big)\\
            &=0.
        \end{align*}
        Hence the identity (\ref{nsp3}) also holds. This completes the proof.
    \end{proof}

    Let $(A, \cdot, \{ ~, ~ \})$ be a Poisson algebra and $N : A \rightarrow A$ be a Nijenhuis operator on it. Then it follows from the above proposition that $(A, \cdot_N, \{ ~ , ~ \}_N)$ is a Poisson algebra, where
    \begin{align*}
        x \cdot_N y = N(x) \cdot y + x \cdot N(y) - N (x \cdot y) ~~~ \text{ and } ~~~ \{x, y \}_N = \{ N(x) , y \} + \{ x, N(y) \} - N \{ x, y \},
    \end{align*}
    for $x, y \in A$. The Poisson algebra $(A, \cdot_N, \{ ~ , ~ \}_N)$ is said to be the {\em Nijenhuis deformation} of $(A, \cdot, \{ ~, ~ \})$ by the Nijenhuis operator $N$.

    It is well-known \cite{kosmann-magri} that a Nijenhuis operator on a commutative associative algebra (resp. Lie algebra) gives rise to a hierarchy of new commutative associative algebra structures (resp. Lie algebra structures) on the underlying vector space. In the following, we generalize this result in the context of Poisson algebras using the help of NS-Poisson algebras. We start with the following notion.

    \begin{definition}
        Let $(A, \ast, \curlyvee, \diamond, \smallblackdiamond)$ and $(A, \ast', \curlyvee', \diamond', \smallblackdiamond')$ be two NS-Poisson algebra structures on a vector space $A$. They are said to be {\bf compatible} if $ ( A, \ast +  \ast', \curlyvee + \curlyvee', \diamond + \diamond', \smallblackdiamond + \smallblackdiamond')$ is also an NS-Poisson algebra.
    \end{definition}

    With the above definition, we have the following simple result.

    \begin{proposition}\label{prop-comp-ns}
        Let $(A, \ast, \curlyvee, \diamond, \smallblackdiamond)$ and $(A, \ast', \curlyvee', \diamond', \smallblackdiamond')$ be compatible NS-Poisson algebra structures on $A$. Then the corresponding sub-adjacent Poisson algebras $(A, \odot, \dcb{~}{~} )$ and $(A, \odot', \dcb{~}{~}')$ are also compatible in the sense that $(A, \odot + \odot' , \dcb{~}{~} + \dcb{~}{~}')$ is a Poisson algebra.
    \end{proposition}

    \begin{proposition}\label{prop-hier}
        Let $(A, \cdot. \{ ~, ~ \})$ be a Poisson algebra and $N: A \rightarrow A$ be a Nijenhuis operator on it. Then we have the following.
        \begin{itemize}
            \item[(i)] For any $k \geq 0$, the map $N^k : A \rightarrow A$ is also a Nijenhuis operator on the Poisson algebra $A$. Hence one obtains the induced NS-Poisson algebra $(A, \ast_{N^k}, \curlyvee_{N^k}, \diamond_{N^k}, \smallblackdiamond_{N^k})$ and the corresponding sub-adjacent Poisson algebra $(A, \cdot_{N^k}, \{ ~, ~ \}_{N^k})$.
            \item[(ii)] For any $k, l \geq 0$, the NS-Poisson algebras $(A, \ast_{N^k}, \curlyvee_{N^k}, \diamond_{N^k}, \smallblackdiamond_{N^k})$ and $(A, \ast_{N^l}, \curlyvee_{N^l}, \diamond_{N^l}, \smallblackdiamond_{N^l})$ are compatible. As a result, the corresponding sub-adjacent Poisson algebras $(A, \cdot_{N^k}, \{ ~, ~ \}_{N^k})$ and $(A, \cdot_{N^l}, \{ ~, ~ \}_{N^l})$ are compatible.
        \end{itemize}
    \end{proposition}

    \begin{proof}
        (i) Since $N : A \rightarrow A$ is a Nijenhuis operator on the commutative associative algebra $(A, \cdot)$ \big(resp. on the Lie algebra $(A, \{ ~, ~ \})$\big), it follows that $N^k$ is also a Nijenhuis operator on $(A, \cdot)$ \big(resp. on $(A, \{ ~, ~ \})$\big) \cite{kosmann-magri}. Hence the result follows.

        (ii) Since $N$ is a Nijenhuis operator on both the commutative associative algebra $(A, \cdot)$ and on the Lie algebra $(A, \{ ~, ~ \})$, the Nijenhuis operators $N^k$ and $N^l$ are compatible Nijenhuis operators on them \cite{kosmann-magri}. In other words, the sum $N^k + N^l$ is also a Nijenhuis operator on both $(A, \cdot)$ and $(A, \{ ~, ~ \})$. The NS-Poisson algebra structure induced by $N^k + N^l$ is given by
        \begin{align*}
            (A, \ast_{N^k} + \ast_{N^l} , \curlyvee_{N^k} + \curlyvee_{N^l}, \diamond_{N^k} + \diamond_{N^l}, \smallblackdiamond_{N^k} +  \smallblackdiamond_{N^l}  ).
        \end{align*}
        This shows that the NS-Poisson algebras $(A, \ast_{N^k}, \curlyvee_{N^k}, \diamond_{N^k}, \smallblackdiamond_{N^k})$ and $(A, \ast_{N^l}, \curlyvee_{N^l}, \diamond_{N^l}, \smallblackdiamond_{N^l})$ are compatible. The last part follows from Proposition \ref{prop-comp-ns}.
    \end{proof}

    \begin{exam}
        Let $A$ be a $4$-dimensional vector space with basis $\{ e_1, e_2, e_3, e_4 \}$. Define a commutative associative multiplication and a Lie bracket on $A$ by
        \begin{align*}
            e_1 \cdot e_1 = e_2, \quad e_1 \cdot e_2 = e_2 \cdot e_1 = e_3, \\
            \{ e_1, e_4 \} = a e_3 + b e_4 = -\{ e_4, e_1 \},
        \end{align*}
        for some constants $a, b \in {\bf k}$. Then $(A, \cdot, \{ ~, ~ \})$ is a Poisson algebra. Moreover, the map $N : A \rightarrow A$ given by
        \begin{align*}
            N (e_1) = r e_1, \quad N (e_2) = r e_2 + s e_3 + t e_4, \quad N (e_3) = r e_3, \quad N (e_4) = r e_4
        \end{align*}
        is a Nijenhuis operator on the Poisson algebra $(A, \cdot, \{ ~, ~ \})$. Thus, by Proposition \ref{prop-nij-nsp}, we have that $(A, \ast, \curlyvee, \diamond, \smallblackdiamond)$ is an NS-Poisson algebra, where
        \begin{align*}
            &e_1 \ast e_1 = r e_2, \quad e_1 \ast e_2 = e_2 \ast e_1 = r e_3, \\
            &e_1 \curlyvee e_1 = -re_2 - s e_3 - t e_4, \quad e_1 \curlyvee e_2 = e_2 \curlyvee e_1 = - r e_3, \\
            & e_1 \diamond e_4 = r a e_3 + r b e_4 = - e_4 \diamond e_1, \quad e_2 \diamond e_1 = - t a e_3 - t b e_4,\\
            & e_1 \smallblackdiamond e_4 = - ra e_3 - rb e_4 = - e_4 \smallblackdiamond e_1.
        \end{align*}
    \end{exam}

    \medskip

In \cite{uchino}, Uchino considered twisted Rota-Baxter operators on associative algebras and showed that they induce NS-algebras. Since Reynolds operators can be realized as twisted Rota-Baxter operators, they also induce NS-algebras. 
These results were generalized to the context of Lie algebras in \cite{das}. Here we will consider twisted Rota-Baxter operators and Reynolds operators on Poisson algebras and show that they induce NS-Poisson algebras. 
First, recall that the cohomology of a Poisson algebra with coefficients in a representation was given by \cite{hueb}. To define twisted Rota-Baxter operators on Poisson algebras, we only require the Poisson $2$-cocycles. Let $(A, \cdot, \{ ~, ~\})$ be a Poisson algebra and $V = (V, \mu, \rho)$ be a representation. A {\bf Poisson $2$-cocycle} of $(A, \cdot, \{ ~, ~\})$ with coefficients in the representation $V$ is a pair $(h, H)$, where

(i) $h \in \mathrm{Hom} (A^{\otimes 2}, V)$ is a $2$-cocycle on the Harrison complex of the commutative associative algebra $(A, \cdot)$ with coefficients in the representation $(V, \mu)$, i.e.
\begin{align}\label{comm-2co}
    h(x, y) = h(y, x) \quad \text{ and } \quad \mu_x h(y, z) + h (x, y \cdot z) = \mu_y h (x, z) + h (y, x \cdot z), \text{ for } x, y, z \in A,
\end{align}

(ii) $H \in \mathrm{Hom}(\wedge^2 A, V)$ is a $2$-cocycle on the Chevalley-Eilenberg complex of the Lie algebra $(A, \{ ~,~\})$ with coefficients in the representation $(V, \rho)$, i.e.
\begin{align}\label{lie-2co}
    \rho_x H (y, z) + \rho_y H (z, x) + \rho_z H (x, y) + H (x, \{ y, z \}) + H (y, \{z, z \}) + H (z, \{ x, y \}) = 0, \text{ for } x, y, z \in A,
\end{align}
that satisfy the following compatibility:
\begin{align}\label{comp-2co}
    H (x, y \cdot z) + \rho_x h (y, z) = h (\{ x, y \}, z) + \mu_z H (x, y) + h (y, \{ x, z \} + \mu_y H (x, z ), \text{ for all } x, y, z \in A.
\end{align}

\begin{remark}\label{remark-exam-2co}
    Let $(A, \cdot, \{ ~, ~\})$ be a Poisson algebra. Then its adjoint representation (also called the regular representation) is given by $(A, \mu, \rho)$, where $\mu_x y = x \cdot y$ and $\rho_x y = \{ x, y \}$, for $x, y \in A$. It is easy to observe that the pair $(h, H) = (\cdot, \{ ~, ~ \})$ is a Poisson $2$-cocycle of $(A, \cdot, \{ ~, ~\})$ with coefficients in the adjoint representation.
\end{remark}

\begin{definition}
    Let $(A, \cdot, \{ ~, ~\})$ be a Poisson algebra and $(V, \mu, \rho)$ be a representation of it. Suppose $(h, H)$ is a Poisson $2$-cocycle of $(A, \cdot, \{ ~, ~\})$ with coefficients in the representation $V$. A linear map $R: V \rightarrow A$ is said to be a {\bf $(h, H)$-twisted Rota-Baxter operator} on $(A, \cdot, \{ ~, ~\})$ with respect to the representation $V$ if it satisfies
    \begin{align}
        R(u) \cdot R(v) =~& R \big(  \mu_{R(u)} v + \mu_{R(v)}u + h ( R(u), R(v)) \big), \label{first-tw}\\
        \{ R(u), R (v) \} =~& R \big(  \rho_{R(u)} v - \rho_{R(v)} u + H ( R(u), R(v)) \big), \label{second-tw}
    \end{align}
    for all $u, v \in V$.
\end{definition}

\begin{thm}\label{thm-twisted-rota-ns}
    Let $R : V \rightarrow A$ be a $(h, H)$-twisted Rota-Baxter operator on $(A, \cdot, \{ ~,~ \})$ with respect to the representation $V$. Then the quintuple $(V, \ast, \curlyvee, \diamond, \smallblackdiamond)$ is an NS-Poisson algebra, where
    \begin{align*}
        u \ast v := ~&\mu_{R(u)} v, \qquad u \curlyvee v := h (R(u), R(v)), \qquad u \diamond v := \rho_{R(u)} v \\
    & \text{ and } ~~~~ u \smallblackdiamond v := H (R(u),R(v)), \text{ for } u, v \in V.
    \end{align*}
\end{thm}

\begin{proof}
    For any $u, v, w \in V$, we have
    \begin{align*}
        u \ast (v \ast w) = \mu_{R(u)} \mu_{R(v)} w = \mu_{R(u) \cdot R(v)} w \stackrel{(\ref{first-tw})}{=} \mu_{R (u \odot v)} w = (u \odot v) \ast v
    \end{align*}
    and
    \begin{align*}
      &  u \ast (v \curlyvee w) + u \curlyvee (v \odot w) - v \ast (u \curlyvee w) - v \curlyvee (u \odot w) \\
       & \stackrel{(\ref{first-tw})}{=} \mu_{R(u)} h ( R(v), R(w)) + h ( R(u), R(v) \cdot R(w)) - \mu_{R(v)} h (R(u), R(w)) - h ( R(v), R(u) \cdot R(w)) = 0.
    \end{align*}
    This shows that $(V, \ast, \curlyvee)$ is an NS-commutative algebra. On the other hand, $R$ satisfies (\ref{second-tw}) implies that $(V, \diamond, \smallblackdiamond)$ is an NS-Lie algebra (see \cite[Proposition 6.7]{das}). Thus, we are only left to verify the identities (\ref{nsp1}), (\ref{nsp2}) and (\ref{nsp3}).

    For any $u, v, w \in V$, we observe that 
    \begin{align*}
     \dcb{u}{v} \ast w - u \diamond (v \ast w) - v \ast (u \diamond w) 
     &= \mu_{ R \dcb{u}{v}} w - \rho_{ R(u)} \rho_{R(v)} w - \mu_{R(v)} \rho_{R(u)} w \\
     &\stackrel{(\ref{second-tw})}{=} \mu_{ \{ R(u), R(v) \} } w -   \rho_{ R(u)} \rho_{R(v)} w - \mu_{R(v)} \rho_{R(u)} w   \stackrel{(\ref{rep-pois1})}{=} 0
    \end{align*}
    and 
    \begin{align*}
        (u \odot v) \diamond w - u \ast (v \diamond w) - v \ast (u \diamond w) 
        &= \rho_{ R (u \odot v)} w - \mu_{R(u)} \mu_{R(v)} w - \mu_{R(v)} \mu_{R(u)} w \\
        &\stackrel{(\ref{first-tw})}{=} \rho_{R(u) \cdot R(v)} w - \mu_{R(u)} \mu_{R(v)} w - \mu_{R(v)} \mu_{R(u)} w \stackrel{(\ref{rep-pois2})}{=} 0 
    \end{align*}
    which proves (\ref{nsp1}) and (\ref{nsp2}). Finally, we have
    \begin{align*}
        &u \smallblackdiamond (v \odot w) + u \diamond (v \curlyvee w) - \dcb{u}{v} \curlyvee w - w \ast (u \smallblackdiamond v) - v \curlyvee \dcb{u}{w} - v \ast (u \smallblackdiamond w) \\
        &= H ( R(u), R(v) \cdot R(w)) + \rho_{R(u)} h ( R(v), R(w)) - h ( \{ R(u), R(v) \}, R(w)) - \mu_{R(w)} H ( R(u), R(v)) \\
        & \qquad \qquad - h ( R(v) , \{ R(u), R(w) \}) - \mu_{R(v)} H ( R(u), R(w))  \stackrel{(\ref{comp-2co})}{=} 0
    \end{align*}
    which verifies (\ref{nsp3}). Hence the result follows.
\end{proof}

In the above theorem, we showed that a twisted Rota-Baxter operator on a Poisson algebra with respect to a representation induces an NS-Poisson algebra structure on the representation space. In the following, we prove the converse. Namely, we show that any NS-Poisson algebra is always induced by a twisted Rota-Baxter operator. Let $(A, \ast, \curlyvee, \diamond, \smallblackdiamond)$ be any NS-Poisson algebra. Consider the subadjacent Poisson algebra $(A, \odot, \dcb{~}{~})$ given in Theorem \ref{prop-subadj-p}. We have also seen that the maps $\mu, \rho : A \rightarrow \mathrm{End} (A)$ given by
\begin{align*}
    \mu_x y = x \ast y ~~~~ \text{ and } ~~~~ \rho_x y = x \diamond y, \text{ for } x, y \in A,
\end{align*}
makes the triple $(A, \mu, \rho)$ into a representation of the subadjacent Poisson algebra $(A, \odot, \dcb{~}{~})$. We now define maps $h \in \mathrm{Hom} (A^{\otimes 2}, A)$ and $H \in \mathrm{Hom} (\wedge^2 A, A)$ by
\begin{align*}
    h (x, y) = x \curlyvee y ~~~~ \text{ and } ~~~~ H (x, y) = x \smallblackdiamond y, \text{ for } x, y \in A.
\end{align*}

\begin{proposition}\label{prop-last}
    With the above notations, the pair $(h, H)$ is a Poisson $2$-cocycle of $(A, \odot, \dcb{~}{~})$ with coefficients in the representation $(A, \mu, \rho)$.

    Moreover, the identity map $\mathrm{Id} : A \rightarrow A$ is a $(h, H)$-twisted Rota-Baxter operator on $(A, \odot, \dcb{~}{~})$ with respect to the representation $(A, \mu, \rho)$. Further, the induced NS-Poisson algebra structure on $A$ coincides with the given one.
\end{proposition}

\begin{proof}
    Since $(A, \ast, \curlyvee)$ is an NS-commutative algebra, it follows that the map $h$ satisfies (\ref{comm-2co}). Similarly, $(A, \diamond, \smallblackdiamond)$ is an NS-Lie algebra implies that the map $H$ satisfies (\ref{lie-2co}). Next, for any $x, y, z \in A$, 
    \begin{align*}
       & H (x, y \odot z) + \rho_x h (y, z) \\
       & = x \smallblackdiamond (y \odot z) + x \diamond (y \curlyvee z) \\
       & \stackrel{(\text{\ref{nsp3})}}{=} \dcb{x}{y} \curlyvee z + z \ast (x \smallblackdiamond y) + y \curlyvee \dcb{x}{z} + y \ast (x \smallblackdiamond z) \\
       & = h ( \dcb{x}{y}, z) + \mu_z H (x,  y) + h (y, \dcb{x}{z}) + \mu_y H (x, z)
    \end{align*}
    which verifies the condition (\ref{comp-2co}). Hence $(h, H)$ is a Poisson $2$-cocycle of $(A, \odot, \dcb{~}{~})$ with coefficients in the representation $(A, \mu, \rho)$.

    For the second part, we observe that 
    \begin{align*}
        \mathrm{Id}(x) \odot \mathrm{Id} (y) = x \odot y =~& x \ast y + y \ast x + x \curlyvee y \\
        =~& \mathrm{Id} \big(  \mu_{ \mathrm{Id} (x)} y + \mu_{\mathrm{Id} (y)} x + h ( \mathrm{Id} (x), \mathrm{Id} (y) )   \big)
    \end{align*}
    and similarly, $\dcb{ \mathrm{Id} (x) }{ \mathrm{Id} (y) } = \mathrm{Id} \big(  \rho_{\mathrm{Id} (x)} y - \rho_{\mathrm{Id} (y)} x + H ( \mathrm{Id} (x) ,\mathrm{Id} (y)) \big)$, for $x, y \in A$. Hence the identity map $\mathrm{Id} : A \rightarrow A$ is a $(h, H)$-twisted Rota-Baxter operator. 

    Finally, if $(A, \ast', \curlyvee', \diamond', \smallblackdiamond')$ is the induced NS-Poisson algebra structure on $A$, then we have
    \begin{align*}
        x \ast' y =~& \mu_{\mathrm{Id}(x)} y = x \ast y, \qquad \quad x \curlyvee' y = h (\mathrm{Id}(x), \mathrm{Id}(y)) = x \curlyvee y, \\
        x \diamond' y =~& \rho_{\mathrm{Id} (x)} y = x \diamond y ~~~ \text{ and } ~~~ x \smallblackdiamond' y = H (\mathrm{Id}(x), \mathrm{Id}(y)) = x \smallblackdiamond y,
    \end{align*}
    for $x, y \in A$. This concludes the last part.
\end{proof}

A Reynolds operator on a Poisson algebra $(A, \cdot, \{ ~, ~ \})$ is a $(h, H)$-twisted Rota-Baxter operator, where $(h, H) = - (\cdot, \{ ~, ~ \})$ is the Poisson $2$-cocycle of $(A, \cdot, \{ ~, ~ \})$ with coefficients in the adjoint representation (see Remark \ref{remark-exam-2co}).
    
    \begin{definition}
       Let $(A, \cdot , \{~,~\})$ be a Poisson algebra. A linear map $R:A \rightarrow A$ is said to be a {\bf Reynolds operator} on $A$ if for any $x, y \in A$,
        \begin{align}
            R(x)\cdot R(y)&= R ( R(x)\cdot y + x \cdot R(y) -R(x)\cdot R(y)), \label{rey1}\\
            \{R(x),R(y)\}&=R( \{R(x),y\} + \{x,R(y)\} -\{R(x),R(y)\}). \label{rey2}
        \end{align}
    \end{definition}

It follows that \cite{uchino,das} a linear map $R: A \rightarrow A$ is a Reynolds operator on the Poisson algebra  $(A, \cdot, \{~,~\})$ if and only if it is a Reynolds operator on both the commutative associative algebra $(A, \cdot)$ and on the Lie algebra $(A, \{ ~, ~\})$. As a consequence of Theorem \ref{thm-twisted-rota-ns}, we have the following result.
    
    \begin{proposition}\label{prop-rey-nsp}
         Let $ ( A, \cdot, \{ ~, ~\})$ be a Poisson algebra and $R: A \rightarrow A$ be a Reynolds operator on $A$. Define bilinear operations $\ast, \curlyvee, \diamond, \smallblackdiamond : A \times A \rightarrow A$ by 
        $$ x \ast y = R(x) \cdot y ,\quad x \curlyvee y = - R(x) \cdot R(y), \quad x \diamond y = \{R(x), y\} \quad \text{ and  } \quad x \smallblackdiamond y = - \{R(x),R(y)\},$$
        for $x, y \in A$. Then $(A, \ast , \curlyvee , \diamond , \smallblackdiamond)$ is an NS-Poisson algebra.
    \end{proposition}

    \begin{exam}
       Let $A$ be a $4$-dimensional vector space with basis \{$e_1, e_2, e_3, e_4$\}. Define a bilinear multiplication and a bracket on $A$ by
        \begin{align*}
            &e_1 \cdot e_1 = e_3, \quad e_1 \cdot e_2 = e_2 \cdot e_1 = e_4,\\
            &\{e_1, e_4\}= -\{e_4, e_1 \} = e_4, \quad \{e_2, e_3 \}= -\{e_3, e_2\} = -2e_4.
        \end{align*}
        Then ($A, \cdot, \{~,~\}$) is a Poisson algebra. We define a linear map $R : A \rightarrow A$ by 
        \begin{align*}
            R(e_1)= ae_1, \quad R(e_3)= \frac{a}{2-a} e_3, \quad R(e_2)= R(e_4)=0,
        \end{align*}
        where $a \neq 2$ is a constant. It is straightforward to check that $R$ is a Reynolds operator on $A$. By Proposition \ref{prop-rey-nsp}, the quintuple $(A, \ast, \curlyvee, \diamond, \smallblackdiamond)$ is an NS-Poisson algebra, where the non-zero products are given by
        \begin{align*}
            &e_1 \ast e_1= a e_3, \quad e_1 \ast e_2 = a e_4,\\
            & e_1 \curlyvee e_1 = -a^2 e_3, \\
            & e_1 \diamond e_4 = a e_4, \quad e_3 \diamond e_2 = \frac{2a}{2-a} e_4.
        \end{align*}
    \end{exam}

\medskip

Since Gerstenhaber algebras \cite{kosmann} are the graded analog of Poisson algebras, one could generalize many results of Poisson algebras to Gerstenhaber algebras. In \cite{aguiar}, Aguiar has considered the notion of a pre-Gerstenhaber algebra as the graded analog of pre-Poisson algebras. In the following, we will introduce NS-Gerstenhaber algebras that simultaneously generalize pre-Gerstenhaber algebras and Gerstenhaber algebras.

    \begin{definition}
        A \textbf{Gerstenhaber algebra} is a triple $(\mathcal{A}, \cdot, \{ ~, ~\})$ consisting of a graded vector space $\mathcal{A} = \oplus_{i \in \mathbb{Z}} \mathcal{A}^i$ with a degree $0$ graded bilinear product $\cdot : \mathcal{A} \times \mathcal{A} \rightarrow \mathcal{A}$ that makes $(\mathcal{A}, \cdot)$ into a graded commutative associative algebra and a degree $-1$ graded bilinear bracket $\{ ~, ~ \} : \mathcal{A} \times \mathcal{A} \rightarrow \mathcal{A}$ satisfying
        \begin{align}\label{gla+1}
            \{x,y\} = -(-1)^{(|x|-1)(|y|-1)} \{y,x \},
            \end{align}
            \begin{align}\label{gla+2}
            (-1)^{(|x|-1)(|z|-1)} \{\{x,y\},z\} + 
            (-1)^{(|y|-1)(|x|-1)} \{\{y, z\},x\} + 
            (-1)^{(|z|-1)(|y|-1)} \{\{z,x\},y\} = 0
         \end{align}
and the following graded Leibniz identity holds:
         $$ \{x, y \cdot z\}= \{x, y \} \cdot z + (-1)^{(|x|-1)|y|} y \cdot \{x,z \},$$
         for all homogeneous elements $x \in \mathcal{A}^{|x|}, y \in \mathcal{A}^{|y|}, z \in \mathcal{A}^{|z|}.$
    \end{definition}

In a Gerstenhaber algebra $(\mathcal{A}, \cdot, \{ ~, ~\})$, the identities (\ref{gla+1}) and (\ref{gla+2}) are equivalent to mean that the shifted graded vector space $\mathcal{A}[+1] = \oplus_{i \in \mathbb{Z}} \mathcal{A}^{i+1}$ with the bracket $\{ ~, ~ \}$  is a graded Lie algebra.
    
    
        Let $(\mathcal{A}, \cdot, \{~,~\})$ be a Gerstenhaber algebra. A {\bf Nijenhuis operator} on $\mathcal{A}$ is a degree $0$ graded linear map $N: \mathcal{A} \rightarrow \mathcal{A}$ that satisfies
        \begin{align*}
            N(x) \cdot N(y) &= N (N(x)\cdot y + x \cdot N(y) - N(x \cdot y)),\\
            \{N(x), N(y)\} &= N(\{N(x), y\} + \{x, N(y)\} - N\{x, y \}), \text{~for all~} x,y \in \mathcal{A}.
        \end{align*}

Before defining NS-Gerstenhaber algebras, we first consider the following notions.

    \begin{definition}
        A \textbf{graded NS-commutative algebra} is a triple $(\mathcal{A}, \ast, \curlyvee)$, where $\mathcal{A} = \oplus_{i \in \mathbb{Z}} \mathcal{A}^i$ is a graded vector space and $\ast , \curlyvee : \mathcal{A} \times \mathcal{A} \rightarrow \mathcal{A}$ are degree $0$ graded bilinear maps in which $\curlyvee$ is graded commutative (i.e. $x\curlyvee y=(-1)^{|x||y|}y\curlyvee x$) and satisfies the following compatibility conditions:
        \begin{align}
            x \ast (y \ast z) =~& (x \odot y ) \ast z,\\
            x \ast (y \curlyvee z) + x \curlyvee ( y \odot z) =~& (-1)^{|y||x|} y \ast (x \curlyvee z) + (-1)^{|y||x|} y \curlyvee(x \odot z ),
        \end{align}
        for $x \in \mathcal{A}^{|x|}, y \in \mathcal{A}^{|y|}$ and $z \in \mathcal{A}^{|z|}$.
        Here we have used $x\odot y= x \ast y + (-1)^{|x||y|}y\ast x + x \curlyvee y$.
    \end{definition}

    \begin{definition}
        A \textbf{graded NS-Lie algebra} is triple $(\mathcal{A}, \diamond, \smallblackdiamond)$, where $\mathcal{A} = \oplus_{i \in \mathbb{Z}} \mathcal{A}^i$ is a graded vector space and $\diamond , \smallblackdiamond : \mathcal{A} \times \mathcal{A} \rightarrow \mathcal{A}$ are degree $0$ graded bilinear maps in which $\smallblackdiamond$ is graded skewsymmetric (i.e. $x\smallblackdiamond y = - (-1)^{|x||y|}y\smallblackdiamond x$) and satisfies the following compatibilities:
        \begin{align}
                 & x \diamond (y \diamond z) - (x \diamond y) \diamond z - (-1)^{|x||y|} (y \diamond (x \diamond z) - (y \diamond x) \diamond z)= (x \smallblackdiamond y) \diamond z,\\
            \begin{split}
                 x \diamond (y \smallblackdiamond z) + x \smallblackdiamond \dcb{y}{z} + 
                (-1)^{|x|(|y|+|z|)}(y\diamond (z \smallblackdiamond x)+y\smallblackdiamond \dcb{z}{x})\\ 
                +(-1)^{|z|(|x|+|y|)}(z\diamond (x\smallblackdiamond y)+ z \smallblackdiamond \dcb{x}{y})=0,
            \end{split}
        \end{align}
        for $x \in \mathcal{A}^{|x|}, y \in \mathcal{A}^{|y|}$ and $z \in \mathcal{A}^{|z|}$.
        Here we used $\dcb{x}{y}= x \diamond y- (-1)^{|x||y|}y\diamond x + x \smallblackdiamond y$.
    \end{definition}
    
    
    \begin{definition}\label{NSGer} An \textbf{NS-Gerstenhaber algebra} is a quintuple $(\mathcal{A}, \ast, \curlyvee, \diamond, \smallblackdiamond)$, where $\mathcal{A} = \oplus_{i \in \mathbb{Z}} \mathcal{A}^i$ is a graded vector space and $\ast, \curlyvee : \mathcal{A} \times \mathcal{A} \rightarrow \mathcal{A}$ are degree $0$ graded bilinear maps that make $(\mathcal{A}, \ast, \curlyvee)$ into a graded NS-commutative algebra, and $\diamond, \smallblackdiamond : \mathcal{A} \times \mathcal{A} \rightarrow \mathcal{A}$ are degree $-1$ graded bilinear maps that make $(\mathcal{A}[+1], \diamond, \smallblackdiamond)$ into a graded NS-Lie algebra satisfying the following compatibilities:
        \begin{align}
            & \dcb{x}{y}\ast z = x\diamond (y\ast z) - (-1)^{(|x|-1)|y|}y \ast (x\diamond z),\\
            & (x\odot y)\diamond z= x\ast (y\diamond z)+ (-1)^{|x||y|}y\ast (x\diamond z),\\
            \begin{split}
             x \smallblackdiamond (y\odot z)+ x\diamond (y\curlyvee z)=\dcb{x}{y}\curlyvee z + (-1)^{|z|(|x|+|y|-1)}z\ast (x\smallblackdiamond y)\\+ (-1)^{(|x|-1)|y|}(y\curlyvee \dcb{x}{z} + y\ast (x\smallblackdiamond z)),
            \end{split}
        \end{align}
        for all homogeneous elements $x \in \mathcal{A}^{|x|},y \in \mathcal{A}^{|y|}$ and $z \in \mathcal{A}^{|z|}$. Here we used
        \begin{align}\label{NSGer-Ger}
        x\odot y = x\ast y+ (-1)^{|x||y|}y\ast x + x \curlyvee y ~ \text{ and } ~ \dcb{x}{y}=x\diamond y - (-1)^{(|x|-1)(|y|-1)}y\diamond x + x \smallblackdiamond y.
        \end{align}
    \end{definition}

    \begin{remark}
        Let $(\mathcal{A}, \ast, \curlyvee, \diamond, \smallblackdiamond)$ be an NS-Gerstenhaber algebra. If $\curlyvee$ and $\smallblackdiamond$ are trivial then $(\mathcal{A}, \ast, \diamond)$ is a pre-Geratenhaber algebra \cite{aguiar}. On the other hand, if $\ast$ and $\diamond$ are trivial then $(\mathcal{A}, \curlyvee, \smallblackdiamond = \{ ~, ~\})$ becomes a Gerstenhaber algebra.
    \end{remark}

        The following results are the graded analog of Theorem \ref{prop-subadj-p} and Proposition \ref{prop-nij-nsp}, respectively. Hence we will not repeat the similar proofs here.

    \begin{proposition}
        Let $(\mathcal{A}, \ast , \curlyvee, \diamond, \smallblackdiamond)$ be an NS-Gerstenhaber algebra. Then the triple $(\mathcal{A}, \odot, \dcb{~}{~} ) $ is a Gerstenhaber algebra, where $\odot$ and $\dcb{~}{~}$ are defined in (\ref{NSGer-Ger}).
    \end{proposition}
    
    \begin{proposition}
        Let $(\mathcal{A}, \cdot , \{~,~\})$ be a Gerstenhaber algebra and $N : \mathcal{A} \rightarrow \mathcal{A}$ be a Nijenhuis operator on $\mathcal{A}$. We define maps $\ast, \curlyvee, \diamond , \smallblackdiamond : \mathcal{A} \times \mathcal{A} \rightarrow \mathcal{A}$ by
        $$ x \ast y = N(x)\cdot y, \quad x \curlyvee y = - N(x \cdot y), \quad x \diamond y = \{N(x), y\} \text{~~ and ~~ } x \smallblackdiamond y = -N\{x,y\},$$
        for $x, y \in \mathcal{A}$. Then $(\mathcal{A}, \ast , \curlyvee , \diamond , \smallblackdiamond)$ is an NS-Gerstenhaber algebra.
    \end{proposition}

\section{NS-Poisson algebras arising from deformations and filtrations of NS-algebras} \label{section4}

In this section, we first define the semi-classical limit of an NS-algebra deformation of a given NS-commutative algebra. We show that the semi-classical limit carries an NS-Poisson algebra structure. This generalizes the well-known semi-classical limit of an associative algebra deformation of a commutative associative algebra. Next, we consider filtrations of an NS-algebra and show that some particular classes of filtrations produce NS-Poisson algebra structures.

    Let $(A, \cdot)$ be a commutative associative algebra. An associative formal deformation of $A$ is given by a ${\bf k}[ \! [ t ] \! ]$-bilinear multiplication $\cdot_t : A[\![t]\!] \times A[\![t]\!] \rightarrow A[\![t]\!]$ of the form
    \begin{align*} x \cdot_t y = x \cdot_0 y + (x \cdot_1 y)t + (x \cdot_2 y)t^2 + (x \cdot_3 y)t^3 + \cdots ~~~ (\text{with }  x \cdot_0 y = x \cdot y, \text{ for } x, y \in A)
    \end{align*}
    that makes $(A[\![t]\!], \cdot_t )$ into an associative algebra over the ring ${\bf k}[ \! [ t ] \! ]$.
    Then we can define a bilinear skewsymmetric bracket $\{ ~, ~ \} : A \times A \rightarrow A$ by
    $$\{x,y\} := \frac{x \cdot_t y - y \cdot_t x}{t} \big|_{t=0} = x \cdot_1 y- y \cdot_1 x, \text{ for } x, y \in A.$$
Then it is well-known that $\{ ~, ~ \}$ is a Lie bracket on $A$ and satisfies the Leibniz rule in the sense of (\ref{leib-rule}). In other words, $(A, \cdot, \{ ~, ~ \})$ is a Poisson algebra, called the {\bf semi-classical limit} of $(A[\![t]\!], \cdot_t)$.

In \cite{aguiar}, Aguiar generalized the above construction in the context of pre-Poisson algebras. Explicitly, he first considered zinbiel algebras (viewed them as commutative dendriform algebras) and their dendriform deformations. Then, using a similar construction as above, he showed that the underlying space inherits a pre-Poisson algebra structure.

   In the following, we generalize both the above constructions by considering NS-commutative algebras and their NS-algebra deformations. For more details about deformations of NS-algebras and their relations with cohomology, we refer \cite{das2}.

   Let $(A, \ast, \curlyvee_0)$ be an NS-commutative algebra. We realize it as an NS-algebra $(A, \prec_0, \succ_0, \curlyvee_0)$ by setting $x \succ_0 y = y \prec_0 x = x \ast y$, for all $x, y \in A$. An NS-algebra deformation of $(A, \ast, \curlyvee_0)$ is a deformation of $(A, \prec_0, \succ_0, \curlyvee_0)$ as an NS-algebra. Thus, it consist of ${\bf k} [ \! [ t ] \! ]$-bilinear maps $\prec_t, \succ_t, \curlyvee_t : A[\![t]\!] \times A[\![t]\!] \rightarrow A[\![t]\!]$ of the form 
   \begin{align*}
          x \prec_t y = \sum_{i=0}^\infty (x \prec_i y) t^i, \quad x \succ_t y = \sum_{i=0}^\infty (x \succ_i y) t^i ~~~ \text{ and } ~~~ x \curlyvee_t y = \sum_{i=0}^\infty (x \curlyvee_i y) t^i
   \end{align*}
   that makes $( A[\![t]\!], \prec_t, \succ_t, \curlyvee_t)$ into an NS-algebra over the ring ${\bf k} [\![t]\!].$ Then we can define bilinear maps $\diamond, \smallblackdiamond : A \times  A \rightarrow A$ by
\begin{align*}
        x \diamond y :=~& \frac{x \succ_t y - y \prec_t x}{t} \Big|_{t=0} = x \succ_1 y - y \prec_1 x,\\
         x \smallblackdiamond y :=~& \frac{x \curlyvee_t y - y \curlyvee_t x}{t} \Big|_{t=0} = x \curlyvee_1 y - y \curlyvee_1 x.
    \end{align*}

    \begin{thm}\label{thm-defor} 
    Let $(A, \ast, \curlyvee_0)$ be an NS-commutative algebra and let $(A[\![t]\!]), \prec_t, \succ_t, \curlyvee_t)$ be an NS-algebra deformation of it. Then with the above notations, $(A, \ast, \curlyvee_0, \diamond, \smallblackdiamond)$ is an NS-Poisson algebra.
    \end{thm}
    
    \begin{proof}
        To prove that $(A, \ast , \curlyvee_0 , \diamond , \smallblackdiamond)$ is an NS-Poisson algebra, we have to show that $(A, \diamond , \smallblackdiamond)$ is an NS-Lie algebra and the identities (\ref{nsp1}), (\ref{nsp2}), (\ref{nsp3}) hold. To show $(A, \diamond, \smallblackdiamond)$ is NS-Lie algebra, we first define the ${\bf k} [\![t]\!]$-bilinear operations $\diamond_t , \smallblackdiamond_t : A  [\![t]\!] \times A  [\![t]\!] \rightarrow A  [\![t]\!]$ by
                \begin{align*}
            x \diamond_t y := x \succ_t y - y \prec_t x &= (x \succ_0 y - y \prec_0 x) + (x \succ_1 y - y \prec_1 x)t + (x \succ_2 y - y \prec_2 x)t^2 + \cdots\\
            &= (x \diamond y)t + (x \succ_2 y - y \prec_2 x)t^2 + \cdots, \\
            x \smallblackdiamond_t y := x \curlyvee_t y - y \curlyvee_t x &= (x \curlyvee_0 y - y \curlyvee_0 x) + (x \curlyvee_1 y - y \curlyvee_1 x)t + (x \curlyvee_2 y - y \curlyvee_2 x)t^2 + \cdots \\
            &= (x \smallblackdiamond y)t + (x \curlyvee_2 y - y \curlyvee_2 x)t^2 + \cdots,
        \end{align*}
         for $x, y \in A$. Then $(A  [\![t]\!], \diamond_t , \smallblackdiamond_t)$ becomes an NS-Lie algebra over the ring ${\bf k}  [\![t]\!]$ (see \cite{das}). We observe that
        \begin{align*}
            (x \diamond_t y) \diamond_t z &= ((x \diamond y) \diamond z)~ t^2 + \mathcal{O}(t^3),\\
            x \diamond_t (y \diamond_t z) &= (x \diamond (y \diamond z))~ t^2 + \mathcal{O}(t^3),\\
            ( x \smallblackdiamond_t y) \diamond_t z &= ((x \smallblackdiamond y) \diamond z)~ t^2 + \mathcal{O}(t^3).
        \end{align*}
        Thus, while writing the expresion (\ref{nsl1}) for the algebra $(A[\![t]\!], \diamond_t , \smallblackdiamond_t)$ and equating the coefficient of $t^2$ in both sides, one obtains the identity (\ref{nsl1}) for the algebra $(A, \diamond, \smallblackdiamond)$. In a similar manner, we can verify the identity (\ref{nsl2}) for the algebra $(A, \diamond, \smallblackdiamond)$.
        The operation $\smallblackdiamond$ is also skewsymmetric. Thus, $(A, \diamond, \smallblackdiamond)$ is an NS-Lie algebra.

        On the other hand, since $(A[\![t]\!], \prec_t, \succ_t, \curlyvee_t)$ is an NS-algebra, we have the following set of identities:
        \begin{align*}
            \sum_{i+j=n} ( (x \prec_i y) \prec_j z )t^n =~& \sum_{i+j=n}( x \prec_i (y \prec_j z + y \succ_j z + y \curlyvee_j z) )t^n,\\
            \sum_{i+j=n} ( (x \succ_i y) \prec_j z )t^n =~& \sum_{i+j=n} ( x \succ_i (y \prec_j z) )t^n,\\        
            \sum_{i+j=n} ( x \succ_i (y \succ_j z) )t^n =~& \sum_{i+j=n}( (x \prec_i y + x \succ_i y + x \curlyvee_i y ) \succ_j z)t^n,\\
            \sum_{i+j=n} ( (x \curlyvee_i y ) \prec_j z ~+~& ( x \prec_i y + x \succ_i y + x \curlyvee_i y) \curlyvee_j z )t^n \\
            =~& \sum_{i+j=n} ( x \succ_i (y \curlyvee_j z)+ x \curlyvee_i ( y \prec_j z + y \succ_j z +y \curlyvee_j z))t^n,          
            \end{align*}
        for all $n \geq 0$ and $x, y, z \in A$.
        For $n=0$, the above four conditions hold as $(A, \succ_0, \prec_0 , \curlyvee_0)$ is an NS-algebra. Next putting $n=1$ and writing $x \succ_0 y = x \ast y$, $x \prec_0 y = y \ast x$, we get the following identities
        \begin{align}
             (y \ast x) \prec_1 z + z \ast (x \prec_1 y) =~& (y \prec_1 z + y \succ_1 z + y \curlyvee_1 z) \ast x + x \prec_1 (z \ast y + y \ast z + y \curlyvee_0 z), \label{1}\\ 
             x \succ_1 (z \ast y ) + x \ast (y \prec_1 z) =~& z \ast (x \succ_1 y) + (x \ast y)\prec_1 z, \label{2}\\
             x \succ_1 (y \ast z ) + x \ast (y \succ_1 z) =~& (x \ast y + y \ast x + x \curlyvee_0 y) \succ_1 z + (x \prec_1 y + x \succ_1 y + x \curlyvee_1 y) \ast z, \label{3}\\
            (x \curlyvee_0 y) \prec_1 z + z \ast (x \curlyvee_1 y) ~+& (x \prec_1 y + x \succ_1 y + x \curlyvee_1 y) \curlyvee_0 z + (y \ast x + x \ast y + x  \curlyvee_0 y) \curlyvee_1 z \label{4} \\
            = x \succ_1 (y \curlyvee_0 z) + x \ast (y ~ \curlyvee_1 & z) + x \curlyvee_0 ( y \prec_1 z + y \succ_1 z + y \curlyvee_1 z) + x \curlyvee_1 (z \ast y +y \ast z + y \curlyvee_0 z). \nonumber
        \end{align}
        We will use these identities to prove the compatibility conditions of the NS-Poisson algebra $(A, \ast, \curlyvee_0, \diamond, \smallblackdiamond)$. First, from (\ref{2}) and (\ref{3}) we get that
        \begin{align*}
             (y \ast z) \prec_1 x-y \ast (z \prec_1 x) =~& y \succ_1 (x \ast z)- x \ast (y \succ_1 z),\\
            (x \succ_1 y + x \prec_1 y + x \curlyvee_1 y) \ast z -x \succ_1(y \ast z)=~& x \ast (y \succ_1 z)- (x \ast y + y \ast x + x \curlyvee_0 y) \succ_1 z,\\
             y \ast (x \succ_1 z)-(y\succ_1 x+ y \prec_1 x+ y \curlyvee_1 x)\ast z =~& (x \ast y + y \ast x + x \curlyvee_0 y) \succ_1 z -y \succ_1 (x \ast z).
        \end{align*}
        By adding these identities, we get
        \begin{align*}
            &(x \succ_1 y + x \prec_1 y + x \curlyvee_1 y- y\succ_1 x- y \prec_1 x- y \curlyvee_1 x )\ast z- (x \succ_1(y \ast z)-(y \ast z) \prec_1 x)\\
            & \qquad \qquad  \qquad \qquad + ( y \ast (x \succ_1 z)-y \ast (z \prec_1 x))=0.
            \end{align*}
            This is equivalent to $(x\diamond y- y\diamond x+ x\smallblackdiamond y)\ast z- x\diamond (y\ast z)+ y\ast (x\diamond z)=0$ or simply
            \begin{align*}
             \dcb{x}{y}\ast z- x\diamond (y\ast z)+ y\ast (x\diamond z)=0. \end{align*}
            Thus, (\ref{nsp1}) holds in $A$. Next, from (\ref{1}), (\ref{2}) and (\ref{3}), we get that
        \begin{align*}
            y \ast (z \prec_1 x)-z\prec_1(x\ast y + y \ast x + x \curlyvee_0 y)=~& (x \prec_1 y + x \succ_1 y + x \curlyvee_1 y)\ast z- (x\ast z)\prec_1 y,\\
             x\ast (z \prec_1 y)- y\ast (x \succ_1 z)=~& (x \ast z)\prec_1 y- x \succ_1 (y\ast z),\\
             (x \ast y+ y \ast x + x \curlyvee_0 y)\succ_1 z- x \ast (y \succ_1 z)=~& x \succ_1 (y \ast z)- (x \prec_1 y + x \succ_1 y + x \curlyvee_1 y)\ast z.
        \end{align*}
        By adding these identities, we obtain 
        \begin{align*}
            &(x \ast y+ y \ast x + x \curlyvee_0 y)\succ_1 z-z\prec_1(x\ast y + y \ast x + x \curlyvee_0 y)-x\ast (y \succ_1 z- z \prec_1 y)\\& \qquad \qquad \qquad \qquad  -y\ast (x \succ_1 z- z \prec_1 x)=0.
            \end{align*}
            This is same as $(x \ast y+ y \ast x + x \curlyvee_0 y)\diamond z- x\ast (y \diamond z)- y \ast (x\diamond z)=0$ or equivalently
            $$(x \odot y)\diamond z -x\ast (y \diamond z)- y \ast (x\diamond z)=0.$$
        This proves (\ref{nsp2}). Finally, from (\ref{4}), we have
        \begin{align*}
            & x \curlyvee_1 (y \ast z + z \ast y + y \curlyvee_0 z) + x \succ_1 (y \curlyvee_0 z) - z \ast (x \curlyvee_1 y)- (x \succ_1 y +x \prec_1 y + x \curlyvee_1 y ) \curlyvee_0 z\\
            & \quad = (x \ast y + y \ast x + x \curlyvee_0 y) \curlyvee_1 z + (x \curlyvee_0 y) \prec_1 z - x \curlyvee_0 (y \prec_1 z + y \succ_1 z + y \curlyvee_1 z) - x \ast (y \curlyvee_1 z),\\
            & y \ast (z \curlyvee_1 x)+ y \curlyvee_0 (z \succ_1 x + z \prec_1 x + z \curlyvee_1 x)-(y \ast z + z \ast y + y \curlyvee_0 z) \curlyvee_1 x - (y \curlyvee_0 z) \prec_1 x  \\
            & \quad = (y \prec_1 z + y \succ_1 z + y \curlyvee_1 z) \curlyvee_0 x + x \ast (y \curlyvee_1 z) - y \curlyvee_1(z \ast x + x \ast z + z \curlyvee_0 x) - y \succ_1 (z \curlyvee_0 x),\\
            & (y \succ_1 x + y \prec_1 x + y \curlyvee_1 x) \curlyvee_0 z+ z \ast (y \curlyvee_1 x) - y \ast (x \curlyvee_1 z)-y \curlyvee_0 (x \succ_1 z + x \prec_1 z + x \curlyvee_1 z) \\
            & \quad = y \curlyvee_1 (x \ast z + z \ast x + x \curlyvee_0 z)+ y \succ_1 (x \curlyvee_0 z)- (y \ast x + x \ast y + y \curlyvee_0 x) \curlyvee_1 z - (y \curlyvee_0 x) \prec_1 z.
        \end{align*}
        By adding these identities, we obtain
        \begin{align*}
            &~x \curlyvee_1 (y \ast z + z \ast y + y \curlyvee_0 z) - (y \ast z + z \ast y + y \curlyvee_0 z) \curlyvee_1 x + x \succ_1 (y \curlyvee_0 z) - (y \curlyvee_0 z) \prec_1 x \\
            & ~~ + z \ast (y \curlyvee_1 x) ~ {- z \ast ( x \curlyvee_1 y)} - y \ast (x \curlyvee_1 z) + y \ast (z \curlyvee_1 x) - (x \succ_1 y + x \prec_1 y + x\curlyvee_1 y) \curlyvee_0 z\\
            & ~~ +(y\succ_1 x+ y \prec_1 x + y \curlyvee_1 x)\curlyvee_0 z-y \curlyvee_0 ( x \succ_1 z + x \prec_1 z + x \curlyvee_1 z) + y \curlyvee_0 ( z \succ_1 x + z \prec_1 x + z \curlyvee_1 x)\\& \qquad =0,
            \end{align*}
            i.e. $x\smallblackdiamond (y\odot z)+ x\diamond (y\curlyvee_0 z)- z\ast (x\smallblackdiamond y) - y\ast (x\smallblackdiamond z)- \dcb{x}{y}\curlyvee_0 z - y\curlyvee_0 \dcb{x}{z}=0.$ In other words, the identity
        (\ref{nsp3}) also holds in $A$. This completes the proof.
\end{proof}

The NS-Poisson algebra $(A, \ast, \curlyvee_0, \diamond, \smallblackdiamond)$ constructed in the above theorem is called the {\bf semi-classical limit} of the NS-algebra deformation $(A [ \! [ t ] \! ], \prec_t, \succ_t, \curlyvee_t)$ of the NS-commutative algebra $(A, \ast, \curlyvee_0)$.

\medskip
 In the following, we consider filtrations of an NS-algebra and show that some particular classes of filtrations produce NS-Poisson algebras. Let $(A, \prec, \succ, \curlyvee)$ be an NS-algebra. A {\bf filtration} of the NS-algebra $A$ is an increasing sequence of subspaces $ A_0 \subseteq A_1 \subseteq A_2  \subseteq \cdots ~$ such that 
    $$ A = \bigcup\limits_{n=0}^{\infty} A_n  \quad \text{ and } \quad  (A_n \succ A_m + A_n \prec A_m ) \subseteq A_{n+m}, \quad (A_n \curlyvee A_m ) \subseteq A_{n+m}. $$
    A filtration $ A_0 \subseteq A_1 \subseteq A_2 \subseteq \cdots ~$ is said to be an {\bf NS-Lie filtration} if it additionally satisfies
        $$ x \succ y - y \prec x \in A_{n+m+1} ~ \text{ and } ~ x \curlyvee y \in A_{n+m+1}, \text{ for } x \in A_{n+1}, y \in A_{m+1}.$$
        
    \begin{thm}\label{thm-filt}
        Let  $(A, \prec, \succ, \curlyvee)$ be an NS-algebra and $ A_0 \subseteq A_1 \subseteq A_2 \subseteq \cdots ~$ be an NS-Lie filtration of it. Then the graded space 
        $$ Gr(A):=\bigoplus_{n=0}^{\infty} \frac{A_{n+1}}{A_n}$$
        inherits an NS-Poisson algebra structure with the operations
        \begin{align*}
            (x + A_n ) &\ast ( y + A_m) = x \succ y + A_{n+m+1},\\
            (x + A_n ) &\curlyvee ( y + A_m) = x \curlyvee y + A_{n+m+1},\\
            (x + A_n ) &\diamond ( y + A_m) = x \succ y - y \prec x + A_{n+m},\\
            (x + A_n ) &\smallblackdiamond ( y + A_m) = x \curlyvee y - y \curlyvee x + A_{n+m},
        \end{align*}
        for $x \in A_{n+1}$ and $y \in A_{m+1}$.
    \end{thm} 
    \begin{proof}

    Let $x + A_n,~ y + A_m,~ z + A_l \in Gr(A).$ According to our assumption, we have $x \succ y - y \prec x \in A_{n+m+1}$. In other words, $x \succ y + A_{n+m+1}= y \prec x + A_{n+m+1}$ which implies $(x+A_n)\succ (y + A_m)= (y + A_m) \prec (x + A_n)$. Moreover, $x \curlyvee y - y \curlyvee x \in A_{n +m+1}$ implies that $(x + A_n) \curlyvee (y + A_m)=(y + A_m) \curlyvee (x + A_n)$. Thus, the associated graded NS-algebra happens to be commutative.
    
    From the definition, it is clear that $\smallblackdiamond$ is a skewsymmetric operation on $Gr(A)$. Thus, with the operations $\diamond$ and $\smallblackdiamond$, the space $Gr(A)$ is an NS-Lie algebra \cite[Proposition 6.5]{das}.

    We now prove the compatibility conditions of NS-Poisson algebra. First, observe that
    \begin{align*}
        &\dcb{x+A_n}{y+A_m} \ast (z+A_l) - (x+A_n) \diamond \big((y+A_m)\ast (z+A_l) \big) + (y+A_m) \ast \big( (x+A_n) \diamond (z+A_l) \big)\\
        &=(x \succ y - y \prec x - y \succ x + x \prec y + x \curlyvee y _ y \curlyvee x) \succ z - x \succ (y \succ z ) + (y \succ z ) \prec x + y \succ (x \succ z)\\
        & \quad - y\succ (z \prec x) + A_{n+m+l+1}\\
        &= \big((x \odot y) \succ z - x \succ (y \succ z)\big ) - \big( (y \odot x) \succ z - y \succ (x \succ z)\big ) + \big( (y \succ z) \prec x - y \succ (z \prec x) \big) + A_{n+m+l+1},
    \end{align*}
    where $x \odot y$ is defined as in Definition \ref{nsalg}. Since $(A, \prec , \succ , \curlyvee)$ is an NS-algebra, the identity (\ref{nsp1}) holds in $Gr(A)$.
Next, we see that
    \begin{align*}
        &\big((x+ A_n) \odot (y + A_m) \big) \diamond (z + A_l) - (x+A_n) \ast \big((y+A_m)\diamond (z+A_l) \big) \\
        & \qquad \qquad - (y+A_m) \ast \big((x+A_n) \diamond (z+A_l) \big)\\
        &= (x \succ y + x \prec y + x \curlyvee y) \succ z - z \prec (x \succ y + x \prec y + x \curlyvee y) - x \succ (y \succ z) + x \succ (z \prec y)
        - (x \succ z) \prec y\\ 
        & \quad + (z \prec x) \prec y + A_{n+m+l+1}\\
        &=\big((x\odot y) \succ z - x \succ (y \succ z) \big)- \big(z \prec (x \odot y)+ (z \prec x)\prec y\big ) + \big(x \succ (z \prec y) - (x \succ z) \prec y \big) + A_{n+m+l+1}\\
        &= A_{n+m+l+1}.
    \end{align*}
    This proves the identity (\ref{nsp2}). Finally, we have
    \begin{align*}
        &\dcb{x+A_n}{y+A_m} \curlyvee (z+A_l) + (z+A_l) \ast \big((x+A_n)\smallblackdiamond (y+A_m)\big) + (y+A_m) \curlyvee \dcb{x+A_n}{z+A_l}\\
        & ~~ + (y + A_m)\ast \big((x+A_n)\smallblackdiamond (z+A_l)\big) - (x+A_n)\smallblackdiamond \big((y+A_m)\odot (z+A_l) \big) \\
        & ~~ - (x+A_n)\diamond \big( (y+A_m)\curlyvee (z+A_l) \big)\\
        &= (x \succ y - y \prec x - y \succ x + x \prec y + x \curlyvee y - y \curlyvee x) \curlyvee z + (x \curlyvee y - y \curlyvee x) \prec z \\
        & ~~ + y \curlyvee (x \succ z - z \prec x - z \succ x + x\prec z + x \curlyvee z - z \curlyvee x) + y \succ (x \curlyvee z - z \curlyvee x)\\
        & ~~ - x \curlyvee (y \succ z + y \prec z + y \curlyvee z) + (y \succ z + y \prec z + y \curlyvee z)\curlyvee x - x \succ (y \curlyvee z) + (y \curlyvee z) \prec x + A_{n+m+l+1}\\
        &= \big( (x \odot y) \curlyvee z + (x \curlyvee y) \prec z - x \curlyvee (y \odot z) - x \succ (y \curlyvee z) \big) - \big( (y \odot x) \curlyvee z + (y \curlyvee x) \prec z - y \curlyvee (x \odot z) \\
        & ~~ - y \succ (x \curlyvee z)\big)
        + \big( (y \odot z) \curlyvee x + (y \curlyvee z) \prec x - y \curlyvee (z \odot x)- y \succ (z \curlyvee x)\big) +A_{n+m+l+1}\\
        &=A_{n+m+l+1}.
    \end{align*}
    Therefore, (\ref{nsp3}) also holds in $Gr(A)$. This proves that $(Gr (A), \ast, \curlyvee, \diamond, \smallblackdiamond)$ is an NS-Poisson algebra.
\end{proof}

\section{NS-{\em F}-manifold algebras} \label{section5}

In this section, we introduce NS-$F$-manifold algebras as a generalization of NS-Poisson algebras, $F$-manifold algebras and pre-$F$-manifold algebras. Any NS-$F$-manifold algebra has a sub-adjacent $F$-manifold algebra. We show that Nijenhuis operators and Reynolds operators on $F$-manifold algebras give rise to NS-$F$-manifold algebras. Finally, we consider the NS-pre-Lie algebra deformation of an NS-commutative algebra and observe that the corresponding semi-classical limit inherits an NS-$F$-manifold algebra structure.

    \begin{definition} 
        An \textbf{$F$-manifold algebra} is a triple $(A, \cdot, \{~,~\})$ in which $(A, \cdot)$ is a commutative associative algebra, $(A, \{~,~\})$ is a Lie algebra and the following Hertling-Manin relation holds:
        \begin{align}
            P_{x \cdot y}(z,w)= x \cdot P_y (z,w) + y \cdot P_x (z,w),
        \end{align}
        for all $x,y,z,w \in A$. Here $P_x (y,z)$ is defined by
        \begin{align*}
            P_x (y,z)= \{x, y\cdot z\}- \{x,y\}\cdot z- y \cdot \{x,z\}, \text{ for } x, y, z \in A.
            \end{align*}
    \end{definition}

    It follows that an $F$-manifold algebra for which $P_x (y, z) = 0$ (for all $x, y, z \in A$) is a Poisson algebra. The notion of $F$-manifold algebra was introduced by Dotsenko \cite{dotsenko} as the underlying algebraic structure of $F$-manifolds. It has been shown by Liu, Sheng and Bai \cite{liu-sheng-bai} that $F$-manifold algebras are the semi-classical limits of pre-Lie formal deformations of commutative pre-Lie algebras.

    Let $(A, \cdot, \{ ~, ~ \})$ be an $F$-manifold algebra. Recall that \cite{liu-sheng-bai} a representation of $(A, \cdot, \{ ~, ~ \})$ is a triple $(V, \mu, \rho)$ in which $(V, \mu)$ is a representation of the commutative associative algebra $(A, \cdot)$ and $(V, \rho)$ is a representation of the Lie algebra $(A, \{ ~, ~ \})$ satisfying additionally
\begin{align*}
    R_{\mu, \rho} (x \cdot y , z) =~& \mu_x \circ R_{\mu, \rho} (y, z) + \mu_y \circ R_{\mu, \rho} (x, z),\\
    \mu_{P_x (y, z)} =~& S_{\mu, \rho} (y, z) \circ \mu_x - \mu_x \circ S_{\mu, \rho} (y, z),
 \end{align*}
 for all $x, y, z \in A$, where $R_{\mu, \rho}, S_{\mu, \rho} : A \otimes A \rightarrow \mathrm{End} (V)$ are defined by
 \begin{align*}
     R_{\mu, \rho} (x, y) := \rho_x \circ \mu_y - \mu_y \circ \rho_x - \mu_{ \{ x, y \} } ~~~ \text{ and } ~~~ S_{\mu, \rho } (x, y) := \mu_x \circ \rho_y + \mu_y \circ \rho_x - \rho_{x \cdot y}.
 \end{align*}

    In \cite{liu-sheng-bai} the authors also introduced the notion of pre-$F$-manifold algebras as a generalization of pre-Poisson algebras. Moreover, they showed that Rota-Baxter operators on $F$-manifold algebras induce pre-$F$-manifold algebras. In the following, we will introduce NS-$F$-manifold algebras as a simultaneous generalization of NS-Poisson algebras, $F$-manifold algebras and pre-$F$-manifold algebras.

    \begin{definition} \label{nsfm}
        An \textbf{NS-$F$-manifold algebra} is a quintuple $(A, \ast, \curlyvee, \diamond, \smallblackdiamond),$ where $(A, \ast, \curlyvee)$ is an NS-commutative algebra and $(A, \diamond, \smallblackdiamond)$ is an NS-Lie algebra such that the following compatibility conditions hold:
        \begin{align*}
             F_1(x \odot y, z, w) =~& x \ast F_1(y,z,w) + y \ast F_1(x,z,w), \label{NSF1} \tag{NSF1}\\
             ( F_1(x,y,z)+ F_1(x,z,y) + F_2(y,z,x) + F_3(x,y,z))\ast w =~& F_2(y,z,x\ast w)- x \ast F_2(y,z,w), \label{NSF2} \tag{NSF2} \\
                 F_3(x \odot y, z,w) + F_2(z,w,y \curlyvee x)=~& x \ast F_3(y,z,w) + y \ast F_3(x,z,w) \label{NSF3} \tag{NSF3}
                \\ + x \curlyvee (F_1(y,z,w)~+~& F_1(y,w,z)+F_2(z,w,y)+F_3(y,z,w))
                \\ + y \curlyvee (F_1(x,z,w)~+~&F_1(x,w,z)+F_2(z,w,x)+F_3(x,z,w)), 
        \end{align*}
        for all $x, y, z,w \in A.$ Here $F_1, F_2, F_3 : \otimes^3A \rightarrow A$ are the maps defined by
        \begin{align*}
            & F_1(x,y,z)= x \diamond(y \ast z)- y \ast (x \diamond z) - \dcb{x}{y}\ast z,\\
            & F_2(x,y,z)= x \ast (y \diamond z) + y \ast (x \diamond z) - (x \odot y) \diamond z,\\
            & F_3(x,y,z)= x \smallblackdiamond (y \odot z) + x \diamond (y \curlyvee z) - z \ast (x \smallblackdiamond y) - y \ast (x \smallblackdiamond z) - y \curlyvee \dcb{x}{z} - \dcb{x}{y}\curlyvee z,
        \end{align*}
        where $x \odot y = x \ast y + y \ast x + x \curlyvee y$ and $\dcb{x}{y}= x \diamond y - y \diamond x + x \smallblackdiamond y.$
    \end{definition}

    \begin{remark}
Let $(A, \ast, \curlyvee, \diamond, \smallblackdiamond)$ be an NS-$F$-manifold algebra. 

(i) If $F_1 = F_2 = F_3  =0$ then it is simply an NS-Poisson algebra introduced in Definition \ref{nspois}. Thus, $F_1, F_2, F_3$ measures to what extent  $(A, \ast, \curlyvee, \diamond, \smallblackdiamond)$ become an NS-Poisson algebra.

(ii) If $\ast$ and $\diamond$ are trivial then it follows that $F_! = F_2  = 0$ and
\begin{align*}
    F_3 (x, y, z) = x \smallblackdiamond (y \curlyvee z) - y \curlyvee (x \smallblackdiamond z) - (x \smallblackdiamond y) \curlyvee z, \text{ for } x, y, z \in A.
\end{align*}
In this case, the identities (\ref{NSF1}), (\ref{NSF2}) hold trivially and the identity (\ref{NSF3}) implies that
\begin{align*}
    F_3 (x \curlyvee y, z, w) = x \curlyvee F_3 (y, z, w) + y \curlyvee F_3 (x, z, w), \text{ for } x, y, z, w \in A.
\end{align*}
Thus, $(A, \curlyvee, \smallblackdiamond = \{ ~, ~\})$ is an $F$-manifold algebra.

(iii) If $\curlyvee$ and $\smallblackdiamond$ are trivial then it follows that $F_3  = 0$ and hence the identity (\ref{NSF3}) holds trivially. In this case, the triple $(A, \ast, \diamond)$ becomes a pre-$F$-manifold algebra in the sense of \cite{liu-sheng-bai}.
\end{remark}

     \begin{thm}\label{thm-subadj-nsf}
    Let $(A, \ast , \curlyvee, \diamond , \smallblackdiamond)$ be an NS-$F$-manifold algebra.
        \begin{itemize}
            \item[(i)] Then $(A, \odot , \dcb{~}{~})$ is an $F$-manifold algebra, called the sub-adjacent $F$-manifold algebra of $(A, \ast , \curlyvee, \diamond , \smallblackdiamond)$ and it is denoted by $A^c$. Here the operations $\odot$ and $\dcb{~}{~}$ are defined in the Definition \ref{nsfm}. 
            \item[(ii)]  Moreover, the triple $(A, \mu, \rho )$ is a representation of the sub-adjacent $F$-manifold algebra $A^c$, where $\mu$ and $\rho$ are defined by (\ref{mu}) and (\ref{rho}), respectively.
        \end{itemize}
    \end{thm}
    \begin{proof}
       (i) Since $(A, \ast, \curlyvee)$ is an NS-commutative algebra, it follows that $(A, \odot)$ is a commutative associative algebra. On the other hand, $(A, \diamond , \smallblackdiamond)$ is an NS-Lie algebra implies that $(A, \dcb{~}{~})$ is a Lie algebra. Thus it remains to verify the Hertling-Manin relation for $(A, \odot, \dcb{~}{~})$. For any $x,y,z \in A$, we first observe that
        \begin{align*}
            P_x(y,z) =&~ \dcb{x}{y \odot z}- \dcb{x}{y} \odot z - y \odot \dcb{x}{z}\\
            =&~ x \diamond (y \ast z + z \ast y + y\curlyvee z) - (y \odot z)\diamond x + x \smallblackdiamond (y \odot z)\\
            &-\dcb{x}{y}\ast z - z \ast (x \diamond y - y \diamond x + x \smallblackdiamond y)- \dcb{x}{y} \curlyvee z\\
            &- y\ast (x \diamond z - z \diamond x + x \smallblackdiamond z)- \dcb{x}{z}\ast y - y \curlyvee \dcb{x}{z}\\
            =&~ \big(x \diamond (y \ast z)- y\ast (x \diamond z)- \dcb{x}{y} \ast z\big) + \big(x \diamond (z \ast y)- z \ast (x \diamond y)- \dcb{x}{z}\ast y\big) \\
            &+\big (y\ast (z \diamond x) + z\ast (y \diamond x)- (y\odot z)\diamond x\big) + \big(x \smallblackdiamond (y \odot z)+ x \diamond (y \curlyvee z) -z\ast (x \smallblackdiamond y)\\ &-y\ast (x \smallblackdiamond z)-y\curlyvee \dcb{x}{z}-\dcb{x}{y}\curlyvee z\big).
        \end{align*}
        Thus, we have
        \begin{align} \label{P}
         P_x(y,z)=  F_1(x,y,z)+F_1(x,z,y)+F_2(y,z,x)+F_3(x,y,z).
        \end{align}
        Hence
        \begin{align*}
            &P_{x\odot y}(z,w)-x \odot P_y(z,w)-y\odot P_x(z,w)\\
            &= F_1(x\odot y, z,w)+F_1(x\odot y,w,z)+F_2(z,w,x\ast y)+F_2(z,w,y\ast x)+F_2(z,w,y\curlyvee x)+F_3(x\odot y, z,w)\\
            & \quad -x\ast \big(F_1(y,z,w)+F_1(y,w,z)+F_2(z,w,y)+F_3(y,z,w)\big)-P_y(z,w)\ast x- x \curlyvee P_y(z,w)\\
            & \quad -y\ast \big(F_1(x,z,w)+F_1(x,w,z)+F_2(z,w,x)+F_3(x,z,w)\big)-P_x(z,w)\ast y - y\curlyvee P_x(z,w)\\
            &= \big(F_1(x\odot y, z,w)-x\ast F_1(y,z,w)-y\ast F_1(x,z,w)\big) +\big(F_1(x\odot y,w,z)-x\ast F_1(y,w,z)\\&  \quad - y\ast F_1(x,w,z)\big)
            -\big(P_y(z,w)\ast x-F_2(z,w,y\ast x)+y\ast F_2(z,w,x)\big)-\big(P_x(z,w)\ast y\\&  \quad  -F_2(z,w,x\ast y)+x\ast F_2(z,w,y)\big)+\big(F_3(x\odot y, z,w)+F_2(z,w,y\curlyvee x)-x\ast F_3(y,z,w)\\&  \quad  -y\ast F_3(x,z,w)- x \curlyvee P_y(z,w)-y\curlyvee P_x(z,w)\big),
        \end{align*}
        which vanishes by using (\ref{P}) along with the compatibility conditions of the NS-$F$-manifold algebra. Therefore, $(A, \odot , \dcb{~}{~})$ is an $F$-manifold algebra.

        (ii) We have already seen that  $(A,\mu)$ is a representation of the commutative associative algebra $(A, \odot)$, and $(A, \rho)$ is a representation of the Lie algebra $(A, \dcb{~}{~})$. Next, observe that 
        \begin{align*}
            F_1(x,y,z)=~&(\rho_x \circ \mu_y - \mu_y \circ \rho_x - \mu_{\dcb{x}{y}})(z)= R_{\mu, \rho }(x,y)(z),\\
            F_2(x,y,z)=~& (\mu_x \circ \rho_y + \mu_y \circ \rho_x - \rho_{x\odot y})(z)= S_{ \mu, \rho}(x,y)(z),
        \end{align*}
        for $x, y, z \in A$. Thus, by using (\ref{NSF1}) we get that
        \begin{align*}
            R_{\mu, \rho}(x\odot y,z)(w)=F_1(x\odot y,z,w)
            &=x\ast F_1(y,z,w)+y\ast F_1(x,z,w)\\
            &=(\mu_x \circ R_{\mu, \rho} (y,z)+\mu_y \circ R_{\mu, \rho}(x,z))(w)
        \end{align*}
        and by using (\ref{NSF2}) and (\ref{P}) we obtain
        \begin{align*}
            \mu_{P_x(y,z)}(w)= P_x(y,z) \ast w
            &= F_2(y,z,x\ast w)-x \ast F_2(y,z,w)\\
            &= (S_{\mu, \rho}(y,z) \circ \mu_x - \mu_x \circ S_{\mu ,\rho}(y,z))(w),
        \end{align*}
        for $x, y, z , w \in A$. This proves that $(A, \mu, \rho)$ is a representation of the sub-adjacent $F$-manifold algebra $A^c$.
    \end{proof}

    In the following, we consider Nijenhuis operators and Reynolds operators on an $F$-manifold algebra and show that they induce NS-$F$-manifold algebra structures. Let $(A, \cdot, \{ ~, ~\})$ be an $F$-manifold algebra. A {\bf Nijenhuis operator} on $(A, \cdot, \{ ~, ~\})$ is a linear map $N : A \rightarrow A$ that is a Nijenhuis operator on both the underlying commutative associative algebra and on the Lie algebra. Similarly, a Reynolds operator on $(A, \cdot, \{ ~, ~ \})$ is a linear map $R: A \rightarrow A$ that is a Reynolds operator on both the underlying commutative associative algebra and on the Lie algebra.

    \begin{proposition}\label{prop-nij-nsf}
        Let $(A, \cdot , \{~,~\})$ be an $F$-manifold algebra and $N:A \rightarrow A$ be a Nijenhuis operator on it. Then $(A, \ast , \curlyvee , \diamond , \smallblackdiamond)$ is an NS-$F$-manifold algebra, where
        $$ x \ast y = N(x)\cdot y, ~~~~ x \curlyvee y= -N(x \cdot y), ~~~~ x \diamond y= \{ N(x),y\} ~~~ \text{ and } ~~~ x \smallblackdiamond y= -N\{x,y\}, ~\text{ for } x, y \in A.$$
    \end{proposition}

    \begin{proof}
        Since $N: A \rightarrow A$ is a Nijenhuis operator on the commutative associative algebra $(A, \cdot)$, it follows that $(A, \ast, \curlyvee)$ is an NS-commutative algebra. Similarly, $N$ is a Nijenhuis operator on the Lie algebra $(A, \{~,~\})$ implies that the triple $(A, \diamond , \smallblackdiamond)$ is an NS-Lie algebra. Next, we observe that
        \begin{align*}
            F_1(x,y,z)
            =~& \{N(x), N(y)\cdot z\}-N(y)\cdot \{N(x),z\}-N\dcb{x}{y} \cdot z \\
            =~& \{N(x), N(y)\cdot z\}-N(y)\cdot \{N(x),z\} - \{N(x),N(y)\}\cdot z\\
            =~& P_{N(x)} (N(y),z) = P_{N(x)} (z, N(y)),\\
            F_2(x,y,z)
            =~& N(x)\cdot \{N(y) , z\}+ N(y)\cdot \{N(x),z\}-\{N(x)\cdot N(y), z \}\\
            =~& P_z(N(x),N(y)) = P_z(N(y), N(x)),\\
            F_3(x,y,z) 
            =~& -N\{x, y \odot z\}-\{N(x), N(y\cdot z)\}+N(z)\cdot N\{x,y\} + N(y)\cdot N \{x,z\}\\ &+ N(y\cdot \dcb{x}{z})+N(\dcb{x}{y}\cdot z)\\
            =~& N \big( -\{x, y\odot z\} + y\odot \{x,z\}+  z\odot \{x,y\} -\dcb{x}{y\cdot z} + y\cdot \dcb{x}{z}+ \dcb{x}{y}\cdot z \big)\\
            =~& N \big( -\{x, N(y)\cdot z\}-\{x, y \cdot N(z)\}+ \{x, N(y\cdot z)\} + N(y) \cdot \{x,z\} + y \cdot N \{x,z\}\\ &- N(y\cdot \{x,z\}) + z\cdot N\{x,y\} + N(z) \cdot \{x,y\} - N(z \cdot \{x,y\})-\{N(x), y\cdot z\}\\ &- \{x, N(y\cdot z)\} + N \{x,y\cdot z\} + y\cdot \{N(x),z\}- y\cdot \{N(z),x\}- y\cdot N\{x,z\}\\
            &+ \{N(x),y\}\cdot z - \{N(y),x\} \cdot z- z\cdot N\{x,y\}  \big)\\
            =~& N \big( -\{x, N(y)\cdot z\} + N(y)\cdot \{x,z\}- \{N(y),x\}\cdot z-\{x, y \cdot N(z)\}+ N(z) \cdot \{x,y\}\\ &-y\cdot \{N(z),x\}-\{N(x), y\cdot z\}+ y\cdot \{N(x),z\}+ \{N(x),y\}\cdot z +N(\{x,y\cdot z\}\\&-y\cdot\{x,z\}-z\cdot \{x,y\})  \big)\\
            =~& -N\big( P_x(N(y),z) + P_x(y,N(z))+P_{N(x)}(y,z)+ N(P_x(y,z))\big )
        \end{align*}
and
        \begin{align*}
            &P_{N(x)}(N(y),N(z))  \\
            &= \{N(x),N(y)\cdot N(z)\}- N(y)\cdot \{N(x), N(z)\} -\{N(x),N(y)\} \cdot N(z)\\
            &= \big\{N(x), N\big(y\cdot N(z)+N(y)\cdot z - N (y\cdot z)\big)\big\} - N(y)\cdot N \big( \{x,N(z)\}+\{N(x),z\} 
            -N\{x,z\} \big) \\ & \quad -N(z)\cdot N \big(\{x,N(y)\}+\{N(x),y\}-N\{x,y\} \big) \\
            &= N\Big(  \{x, N(y)\cdot N(z)\}+\big\{N(x), y\cdot N(z)+N(y)\cdot z-N(y\cdot z)\big\}-N\big\{x, y\cdot N(z) + N(y)\cdot z-N(y\cdot z)\big\} \\
             & \quad -y\cdot \{N(x),N(z)\}-N(y)\cdot \{x,N(z)\}-N(y)\cdot \{N(x),z\} +N(y)\cdot N\{x,z\} \\
             & \quad + N\big(y\cdot \{x,N(z)\}+y\cdot \{N(x),z\}-y\cdot N\{x,z\}\big) 
              -z \cdot \{N(x),N(y)\} - N(z)\cdot \{x,N(y)\}\\ & \quad -N(z)\cdot \{N(x),y\} +N(z)\cdot N\{x,y\}
             +N\big(z\cdot \{x,N(y)\}+z\cdot \{N(x),y\}-z\cdot N\{x,y\}\big)  \Big) \\
             &= N\Big(\{x,N(y)\cdot N(z)\}-N(y)\cdot \{x,N(z)\}- N(z)\cdot \{x,N(y)\}+\{N(x), y\cdot N(z)\} - y\cdot \{N(x),N(z)\} \\
            & \quad -N(z)\cdot \{N(x),y\}+\{N(x),N(y)\cdot z\}-z\cdot \{N(x),N(y)\} - N(y)\cdot \{N(x),z\} -\{N(x), N(y\cdot z)\} \\
            & \quad + N(y)\cdot N\{x,z\}+N(z)\cdot N\{x,y\} - N\big(\{x,y\cdot N(z)\}+\{x,N(y)\cdot z\}-\{x,N(y\cdot z)\} \\
            & \quad -y\cdot \{x,N(z)\}-y\cdot \{N(x),z\}  +y\cdot N\{x,z\}-z\cdot\{x,N(y)\}-z\cdot \{N(x),y\}+z\cdot N\{x,y\}    \big)    \Big)
        \end{align*}
           \begin{align*}  
            &= N \Big( P_x(N(y),N(z))+P_{N(x)}(y,N(z))+P_{N(x)}(N(y),z)- N \big( \{x,N(y\cdot z)\} +\{N(x),y\cdot z\} \\
            & \quad -N\{x,y\cdot z\} -y\cdot N\{x,z\}-N(y)\cdot \{x,z\}+N(y\cdot\{x,z\}) -z\cdot N\{x,y\}-N(z)\cdot \{x,y\}\\
            & \quad +N(z\cdot \{x,y\})+\{x,y\cdot N(z)\}+\{x, N(y)\cdot z \} -\{x, N(y\cdot z)\}\\& 
            \quad -y\cdot \{x,N(z)\}-y\cdot \{N(x),z\}+ y\cdot N\{x,z\} - z\cdot \{x,N(y)\} -z\cdot \{N(x),y\}+ z\cdot N\{x,y\}    \big)  \Big)\\
            &= N \Big( P_x(N(y),N(z))+P_{N(x)}(y,N(z))+P_{N(x)}(N(y),z)- N\big( P_x(N(y),z)  +P_x(y,N(z)) \\ & \quad+P_{N(x)}(y,z)-N(P_x(y,z))\big) \Big)\\
            &= N(F_1(x,y,z)+F_1(x,z,y)+F_2(y,z,x)+F_3(x,y,z)).
        \end{align*}
Using these observations, we will now prove the identities (\ref{NSF1}), (\ref{NSF2}) and (\ref{NSF3}). First, for any $x,y,z,w \in A$, we have
        \begin{align*}
            F_1(x \odot y, z, w) = P_{N(x\odot y)}(N(z),w)
            &= P_{N(x)\cdot N(y)}(N(z),w)\\
            &= N(x) \cdot P_{N(y)}(N(z),w) + N(y)\cdot P_{N(x)}(N(z),w)\\
            &= x\ast F_1(y,z,w) + y \ast F_1(x,z,w),
        \end{align*}
        which verifies the identity (\ref{NSF1}). On the other hand,
        \begin{align*}
            &\big(F_1(x,y,z)+F_1(x,z,y)+F_2(y,z,x)+F_3(x,y,z)\big)\ast w\\
            &= N \big( F_1(x,y,z)+F_1(x,z,y)+F_2(y,z,x)+F_3(x,y,z) \big) \cdot w\\
            &= P_{N(x)}(N(y),N(z))\cdot w\\
            &= P_{N(x)\cdot w}(N(y),N(z))-N(x)\cdot P_w(N(y),N(z))\\
            &= P_{x\ast w}(N(y),N(z))-x\ast P_{w}(N(y),N(z))\\
            &= F_2(y,z,x\ast w)-x\ast F_2(y,z,w),
        \end{align*}
        which proves (\ref{NSF2}). Finally, to prove the identity (\ref{NSF3}), we first observe that
        \begin{align*}
            P_{x\odot y}(z,w)
            &= P_{N(x)\cdot y+x\cdot N(y)-N(x\cdot y)}(z,w)\\
            &= N(x)\cdot P_y(z,w)+ y\cdot P_{N(x)}(z,w)+x\cdot P_{N(y)}(z,w)+N(y)\cdot P_x(z,w)-P_{N(x\cdot y)}(z,w).
        \end{align*}
        Using this, we can calculate $F_3(x\odot y, z,w)$ as
        \begin{align}\label{F3xyzw}
            &F_3(x\odot y, z,w)  \nonumber \\
            &= -N \big( P_{x\odot y}(N(z),w)+P_{x\odot y}(z,N(w))+P_{N(x)\cdot N(y)}(z,w)-N(P_{x\odot y}(z,w))  \big)  \nonumber  \\ 
            &= N \Big( N(x)\cdot P_y(N(z),w)+y\cdot P_{N(x)}(N(z),w)+x\cdot P_{N(y)}(N(z),w)+N(y)\cdot P_x(N(z),w)  \nonumber \\
            & \quad + N(x)\cdot P_y(z,N(w))+y\cdot P_{N(x)}(z, N(w))+x\cdot P_{N(y)}(z,N(w))+N(y)\cdot P_x(z,N(w))  \nonumber \\
            & \quad - P_{N(x\cdot y)}(N(z),w)-P_{N(x\cdot y)}(z,N(w))+N(x)\cdot P_{N(y)}(z,w)+N(y)\cdot P_{N(x)}(z,w)  \nonumber \\ 
            & \quad - N\big( N(x)\cdot P_y(z,w)+ y\cdot P_{N(x)}(z,w)+x\cdot P_{N(y)}(z,w)+N(y)\cdot P_x(z,w)-P_{N(x\cdot y)}(z,w)\big)  \Big)  \nonumber \\
            &= -N \Big( N(x)\cdot \big(P_y(N(z),w)+P_y(z,N(w))+P_{N(y)}(z,w)\big)+N(y)\cdot \big(P_x(N(z),w)+P_x(z,N(w))\\ 
            & \quad +P_{N(x)}(z,w)\big)+x\cdot \big(P_{N(y)}(N(z),w)+P_{N(y)}(z,N(w))\big)+y\cdot \big(P_{N(x)}(N(z),w)+P_{N(x)}(z,N(w))\big)  \nonumber \\
            & \quad - P_{N(x\cdot y)}(N(z),w)-P_{N(x\cdot y)}(z,N(w)) -N\big( N(x)\cdot P_y(z,w)+ y\cdot P_{N(x)}(z,w)+x\cdot P_{N(y)}(z,w)  \nonumber \\
            & \quad +N(y)\cdot P_x(z,w)-P_{N(x\cdot y)}(z,w)\big) \Big).  \nonumber 
        \end{align}
        We also have
        \begin{align}\label{F2zwyx}
            &F_2(z,w,y\curlyvee x) \nonumber \\
            &= P_{-N(y\cdot x)}(N(z),N(w))  \nonumber  \\
            &= -N \big( F_1(x\cdot y,z,w)+F_1(x\cdot y,w,z)+F_2(z,w,x\cdot y)+F_3(x\cdot y,z,w)  \big)  \nonumber  \\
            &= -N \Big( P_{N(x\cdot y)}(N(z),w)+P_{N(x\cdot y)}(N(w),z)+x\cdot P_y(N(z),N(w))+y\cdot P_x(N(z),N(w))\\
            & \quad -N\big( x\cdot (P_y(N(z),w)+P_y(z,N(w)))+y\cdot (P_x(N(z),w)+P_x(z,N(w)))+P_{N(x\cdot y)}(z,w)  \nonumber \\
            & \quad - N( x\cdot P_y(z,w)+y\cdot P_x(z,w) ) \big)  \Big).  \nonumber 
        \end{align}
        By adding the expressions (\ref{F3xyzw}) and (\ref{F2zwyx}), we get the LHS of (\ref{NSF3}) as
        \begin{align} \label{lhs}
            &F_3(x\odot y, z,w)+F_2(z,w,y\curlyvee x)  \nonumber \\
            &= -N \Big( N(x)\cdot \big(P_y(N(z),w)+P_y(z,N(w))+P_{N(y)}(z,w)\big)+N(y)\cdot \big(P_x(N(z),w)+P_x(z,N(w))  \\
            & \quad +P_{N(x)}(z,w)\big)+x\cdot \big(P_{N(y)}(N(z),w)+P_{N(y)}(z,N(w))+P_y(N(z),N(w))  \big)+ y\cdot \big(P_{N(x)}(N(z),w)  \nonumber \\
            & \quad +P_{N(x)}(z,N(w)) +P_x(N(z),N(w))  \big)- N\big(x\cdot (P_y(z,N(w)) 
            +P_y(N(z),w)+P_{N(y)}(z,w)  )  \nonumber  \\
            & \quad +y\cdot(P_x(N(z),w)+P_x(z,N(w))+P_{N(x)}(z,w) ) +N(x)\cdot P_y(z,w)+ N(y)\cdot P_x(z,w)  \nonumber \\
            & \quad  -N(x\cdot P_y(z,w)+y\cdot P_x(z,w) ) \big) \Big).  \nonumber 
        \end{align}
        Similarly, we calculate the following expressions:
        \begin{align*}
           &x\ast F_3(y,z,w)\\
            &= N(x)\cdot F_3(y,z,w)\\
            &= -N \Big( -x\cdot F_3(y,z,w)+N(x)\cdot \big(P_y(N(z),w)+P_y(z,N(w))+P_{N(y)}(z,w)  \big)-N(x)\cdot N(P_y(z,w))\\& ~~~ \quad - N \big( x\cdot (P_y(N(z),w)+P_y(z,N(w))+P_{N(y)}(z,w)) -x\cdot N(P_y(z,w))    \big)  \Big)\\
            &= -N \Big( -x\cdot F_3(y,z,w)+ N(x)\cdot \big(P_y(N(z),w)+P_y(z,N(w))+P_{N(y)}(z,w) \big)-N \big( N(x)\cdot P_y(z,w)\\ & \quad +x\cdot (P_y(N(z),w)+P_y(z,N(w))+P_{N(y)}(z,w))-N(x\cdot P_y(z,w)) \big) \Big),
        \end{align*}
        \begin{align*}
                &x\curlyvee (F_1(y,z,w)+F_1(y,w,z)+F_2(z,w,y)+F_3(y,z,w))\\
                &=-N \big( x\cdot(P_{N(y)}(N(z),w)+P_{N(y)}(N(w),z)+P_y(N(z),N(w))  )+x\cdot F_3(y,z,w) \big).
        \end{align*}
        With the above two expressions, we now calculate the RHS of (\ref{NSF3}) as
        \begin{align}\label{rhs}
                &x\ast F_3(y,z,w)+y\ast F_3(x,z,w)+ x\curlyvee (F_1(y,z,w)+F_1(y,w,z)+F_2(z,w,y)+F_3(y,z,w)) \nonumber \\ &+ y\curlyvee (F_1(x,z,w)+F_1(x,w,z)+F_2(z,w,x)+F_3(x,z,w))  \nonumber  \\
                &= -N \Big( N(x)\cdot \big(P_y(N(z),w)+P_y(z,N(w))+P_{N(y)}(z,w)\big)+N(y)\cdot \big(P_x(N(z),w)+P_x(z,N(w))   \\
                & \quad +P_{N(x)}(z,w)\big)+x\cdot \big(P_{N(y)}(N(z),w)+P_{N(y)}(z,N(w))+P_y(N(z),N(w))  \big)+ y\cdot \big(P_{N(x)}(N(z),w)  \nonumber  \\
                & \quad +P_{N(x)}(z,N(w)) +P_x(N(z),N(w))  \big)- N\big(x\cdot (P_y(z,N(w)) 
                +P_y(N(z),w)+P_{N(y)}(z,w))  \nonumber \\
            & \quad +y\cdot(P_x(N(z),W)+P_x(z,N(w))+P_{N(x)}(z,w) ) +N(x)\cdot P_y(z,w)+ N(y)\cdot P_x(z,w)  \nonumber \\
            &  \quad -N(x\cdot P_y(z,w)+y\cdot P_x(z,w) )    \big)  \Big).  \nonumber
        \end{align}
        Comparing the expressions in (\ref{lhs}) and (\ref{rhs}) shows that the identity (\ref{NSF3}) holds. This completes the proof.
\end{proof}

Let $(A, \cdot, \{ ~, ~\})$ be an $F$-manifold algebra and $N: A \rightarrow A$ be a Nijenhuis operator on it. Then it follows from Proposition \ref{prop-nij-nsf} and Theorem \ref{thm-subadj-nsf} that $(A, \cdot_N, \{ ~, ~\}_N)$ is also an $F$-manifold algebra, where
\begin{align*}
    x \cdot_N y := N(x) \cdot y + x \cdot N(y) - N(x \cdot y) ~~~~ \text{ and }  ~~~~ \{ x, y \}_N := \{ N(x), y \} + \{ x, N(y) \} - N \{ x, y \},
\end{align*}
for $x, y \in N$. The $F$-manifold algebra $(A, \cdot_N , \{ ~, ~\}_N)$ is called the {\em Nijenhuis deformation} of $(A, \cdot, \{ ~, ~\})$ by the Nijenhuis operator $N$. In similar to Proposition \ref{prop-hier}, here we can also show that $N^k : A \rightarrow A$ is a Nijenhuis operator on the $F$-manifold algebra $(A, \cdot, \{ ~, ~\})$, for any $k \geq 0$. Moreover, for any $k, l \geq 0$, the $F$-manifold algebras $(A, \cdot_{N^k} , \{ ~, ~ \}_{N^k})$ and   $(A, \cdot_{N^l} , \{ ~, ~ \}_{N^l})$ are compatible in the sense that $(A, \cdot_{N^k} + \cdot_{N^l}, \{ ~, ~ \}_{N^k} + \{ ~, ~ \}_{N^l})$ is also an $F$-manifold algebra.

\begin{exam}
Let $A$ be a $2$-dimensional vector space with basis $\{ e_1, e_2 \}$. We define a commutative associative multiplication on $A$ by
\begin{align*}
    e_1 \cdot e_1 = e_1 ~~~~  \text{ and } ~~~~ e_1 \cdot e_2 = e_2 \cdot e_1 = e_2.
\end{align*}
A Lie bracket $\{ ~, ~ \}$ on $A$ is defined by $\{ e_1, e _1 \} = \{ e_2, e_2 \} = 0$ and $\{ e_1, e_2 \} = ae_2 = - \{ e_2, e_1 \},$  for any constant $a \in {\bf k}$. It has been shown in \cite{liu-sheng-bai} that $(A, \cdot, \{ ~, ~ \})$ is an $F$-manifold algebra. By considering the identity map $\mathrm{Id} : A \rightarrow A$ as a Nijenhuis operator, we get that
$(A, \ast, \curlyvee, \diamond, \smallblackdiamond)$ is an NS-$F$-manifold algebra, where
\begin{align*}
    &e_i \ast e_j = e_i \cdot e_j, \quad e_i \curlyvee e_j = - e_i \cdot e_j, \text{ for } i, j =1,2, \\
   & e_1 \diamond e_1 = 0, \quad e_1 \diamond e_2 = a e_2, \quad  e_2 \diamond e_1 = - a e_2, \quad e_2 \diamond e_2 = 0,\\
   & e_1 \smallblackdiamond e_1 = 0, \quad e_1 \smallblackdiamond e_2 = - a e_2, \quad e_2 \smallblackdiamond e_1 =  a e_2, \quad  e_2 \smallblackdiamond e_2 = 0.
\end{align*}
More generally, the map $N : A \rightarrow A$ given by $N (e_1) = r e_1 + s e_2$ and $N (e_2 ) = r e_2$ is a Nijenhuis operator on the $F$-manifold algebra $(A, \cdot, \{ ~, ~ \})$, for any constants $r, s \in {\bf k}$. Hence by Proposition \ref{prop-nij-nsf}, $(A, \ast, \curlyvee, \diamond, \smallblackdiamond)$ is an NS-$F$-manifold algebra, where
\begin{align*}
    e_i \ast e_j = N (e_i) \cdot e_j, \quad e_i \curlyvee e_j = - N (e_i \cdot e_j), \quad e_i \diamond e_j = \{ N (e_i), e_j \},  \quad e_i \smallblackdiamond e_j = - N \{ e_i, e_j \}.
\end{align*}
\end{exam}

\begin{exam}
    Let $A$ be a $3$-dimensional vector space with basis $\{ e_1, e_2, e_3 \}$. Then $(A, \ast, \curlyvee, \diamond, \smallblackdiamond)$ is an NS-$F$-manifold algebra, where the bilinear operations $\ast, \curlyvee, \diamond, \smallblackdiamond$ are characterized by 
    \begin{align*}
       & e_2 \ast e_3 = e_3 \ast e_2 = e_1, \quad e_3 \ast e_3 = e_2, \\
       & e_2 \curlyvee e_3 = e_3 \curlyvee e_2 = - e_1, \quad e_3 \curlyvee e_3 = -e_2, \\
       & e_2 \diamond e_3 = a e_1, \quad e_3 \diamond e_2 = - a e_1, \\
       & e_2 \smallblackdiamond e_3 = - a e_1, \quad e_3 \smallblackdiamond e_2 = a e_1.
    \end{align*}
\end{exam}

\begin{exam}
    Let $(A, \cdot)$ be a commutative associative algebra and $D : A \rightarrow A$ be a derivation. Then it has been shown in \cite{liu-sheng-bai} that $(A, \cdot, \{ ~, ~ \})$ is an $F$-manifold algebra, where
    \begin{align*}
        \{ x, y \} := x \cdot D(y) - y \cdot D(x), \text{ for } x, y \in A.
    \end{align*}
    Suppose $N: A \rightarrow A$ is a Nijenhuis operator on the commutative associative algebra $(A, \cdot)$ that satisfies $N \circ D = D \circ N$. Then $N$ is also a Nijenhuis operator on the Lie algebra $(A, \{ ~, ~ \})$. To see this, we observe that
    \begin{align*}
        \{ N(x), N(y) \} =~& N(x) \cdot DN (y) - N(y) \cdot DN (x) \\
        =~& N(x) \cdot N D(y) - N(y) \cdot ND(x) \\
        =~& N \big(   N(x) \cdot D(y) + x \cdot ND (y) - N (x \cdot D(y))  \big)  \\
        ~& - N \big(   N(y) \cdot D(x) + y \cdot ND (x) - N (y \cdot D(x))  \big) \\
        =~& N \big(   \{ N (x), y \} + \{ x, N(y) \} - N \{ x, y \} \big).
    \end{align*}
    Thus, $N$ is a Nijenhuis operator on the $F$-manifold algebra $(A, \cdot, \{ ~, ~ \})$. Therefore, we obtain an NS-$F$-manifold algebra $(A, \ast, \curlyvee, \diamond, \smallblackdiamond)$, where
    \begin{align*}
        x \ast y =~& N(x) \cdot y, \qquad x \curlyvee y = - N (x \cdot y), \qquad x \diamond y = N (x) \cdot D(y) - y \cdot DN (x) \\
        &~~~~ \text{ and } \qquad x \smallblackdiamond y = - N (x \cdot D(y) - y \cdot D(x)), \text{ for } x, y \in A.
    \end{align*}
\end{exam}

\begin{exam}
Let $A = {\bf k}[x]$ be the polynomial algebra in one-variable and $D : A \rightarrow A$ be the derivation given by $D ( p(x) ) = p' (x)$ the derivative of $p(x)$. Consider the identity map $\mathrm{Id} : A \rightarrow A$ which is a Nijenhuis operator on $A$ satisfying $\mathrm{Id} \circ D = D \circ \mathrm{Id}$. Hence we obtain an NS-$F$-manifold algebra $(A,\ast, \curlyvee, \diamond, \smallblackdiamond )$ given by
\begin{align*}
    p (x) \ast q(x) =~& p(x) q(x), \quad p(x) \curlyvee q(x) = - p(x) q(x), \quad p(x) \diamond q(x) = p(x) q'(x) - q(x) p'(x) \\
    &\text{ and } \quad p (x) \smallblackdiamond q(x) = - (p(x) q'(x) - q(x) p'(x)), \text{ for } p(x), q(x) \in A.
\end{align*}
\end{exam}

    \begin{proposition}\label{prop-rey-nsf}
        Let $(A, \cdot , \{~,~\})$ be an $F$-manifold algebra and $R:A \rightarrow A$ be a Reynolds operator on it. Then $(A, \ast, \curlyvee, \diamond, \smallblackdiamond)$ is an NS-$F$-manifold algebra, where
        $$ x\ast y= R(x)\cdot y, ~~ x \curlyvee y = -R(x)\cdot R(y), ~~ x\diamond y= \{R(x),y\} \text{~~and~~} x\smallblackdiamond y= - \{R(x),R(y)\}, \text{ for } x, y \in A.$$
        Then $(A,\ast , \curlyvee, \diamond, \smallblackdiamond)$ is an NS-$F$-manifold algebra.
        \begin{proof}
            We have already seen in Proposition \ref{prop-rey-nsp} that $(A, \ast, \curlyvee)$ is an NS-commutative algebra and $(A, \diamond, \smallblackdiamond)$ is an NS-Lie algebra. By direct calculations, we can observe that
            \begin{align*}
                F_1(x,y,z)&=P_{R(x)}(R(y),z),\\
                F_2(x,y,z)&=P_z(R(x),R(y)),\\
                F_3(x,y,z)&=-2P_{R(x)}(R(y),R(z)),\\
                P_{R(x)}(R(y),R(z)) &=R \big( F_1(x,y,z)+F_1(x,z,y)+F_2(y,z,x)+F_3(x,y,z) \big),
            \end{align*}
            for $x, y, z \in A.$ Using these observations along with the definition of the Reynolds operator and $F$-manifold algebra, we obtain that
            \begin{align*}
                x\ast F_1(y,z,w)+y\ast F_1(x,z,w)&=R(x)\cdot P_{R(y)}(R(z),w)+R(y)\cdot P_{R(x)}(R(z),w)\\
                &= P_{R(x)\cdot R(y)}(R(z),w)\\
                &= P_{x \odot y}(R(z),w)\\
                &= F_1(x\odot y, z,w),
            \end{align*}
            \begin{align*}
                F_2(y,z, x\ast w)-x\ast F_2(y,z,w) &= 
                P_{R(x)\cdot w}(R(y), R(z))-R(x)\cdot P_w (R(y),R(z))\\
                &= P_{R(x)}(R(y), R(z))\cdot w\\
                &= (F_1(x,y,z)+F_1(x,z,y)+F_2(y,z,x)+F_3(x,y,z))\ast w
            \end{align*}
            and
            \begin{align*}
                 & F_3(x \odot y, z,w) + F_2(z,w,y \curlyvee x)- x \ast F_3(y,z,w) - y \ast F_3(x,z,w)
                \\&- x \curlyvee (F_1(y,z,w)+ F_1(y,w,z)+F_2(z,w,y)+F_3(y,z,w))
                \\&- y \curlyvee (F_1(x,z,w)+F_1(x,w,z)+F_2(z,w,x)+F_3(x,z,w))\\
                &= -2P_{R(x)\cdot R(y)}(R(z),R(w))+P_{-R(y)\cdot R(x)}(R(z),R(w))+2R(x)\cdot P_{R(y)}(R(z),R(w))\\
                & \quad +2R(y)\cdot P_{R(x)}(R(z),R(w))+R(x)\cdot P_{R(y)}(R(z),R(w))+ R(y)\cdot P_{R(x)}(R(z),R(w))\\
                &= 0,
            \end{align*}
            which verifies (\ref{NSF1}), (\ref{NSF2}) and (\ref{NSF3}), respectively. Hence the result follows.
        \end{proof}
    \end{proposition}

\begin{remark}
    In future, we intend to consider the cohomology theory of $F$-manifold algebras. In particular, by using $2$-cocycles, we aim to define twisted Rota-Baxter operators and prove the analogs of Theorem \ref{thm-twisted-rota-ns} and Proposition \ref{prop-last} in the context of $F$-manifold algebras.
    \end{remark}

    Next, we consider NS-pre-Lie algebra deformations of NS-commutative algebras and show that NS-$F$-manifold algebras are the corresponding semi-classical limits. We start with the following definition from \cite{guo-zhang}.

    \begin{definition}  \label{nspl}
        An \textbf{NS-pre-Lie algebra} is a quadruple $(A, \rhd, \lhd, \circ)$ of vector space $A$ together with bilinear maps $\rhd, \lhd, \circ : A \times A \rightarrow A$ satisfying the following set of identities:
        \begin{align}
             (x\odot y)\rhd z- x \rhd(y\rhd z) =~& (y\odot x)\rhd z - y\rhd (x\rhd z) \label{nspl1},\\
             x \rhd (y \lhd z)-(x \rhd y)\lhd z =~& y\lhd (x\odot z)-(y \lhd x)\lhd z \label{nspl2},\\
             \begin{split}
                (x\odot y)\circ z-x\circ (y\odot z)+ ~& (x\circ y)\lhd z- x \rhd (y \circ z) \label{nspl3} \\ =~& (y\odot x)\circ z-y\circ (x\odot z) +(y\circ x)\lhd z- y \rhd (x \circ z),
                \end{split}
        \end{align}
        for all $x,y,z \in A$. Here $x\odot y= x\rhd y + x\lhd y + x \circ y$.
    \end{definition}
    \begin{remark}
        Any NS-algebra $(A, \prec, \succ , \curlyvee)$ can be regarded as an NS-pre-Lie algebra $(A,  \rhd, \lhd, \circ)$ by setting $x\rhd y= x \succ y, x \lhd y= x \prec y$ and $x\circ y= x \curlyvee y$. Thus NS-pre-Lie algebras generalize NS-algebras. Taking the product $\circ$ as trivial in any NS-pre-Lie algebra gives us an L-dendriform algebra \cite{bai-l}. So, NS-pre-Lie algebras are also a generalization of L-dendriform algebras. 
    \end{remark}

    It is well-known that a pre-Lie algebra gives rise to a Lie algebra by skew-symmetrization. The following result is a generalization for NS-pre-Lie algebras.

    \begin{proposition}\label{prop-nspre-nsl}
        Let $(A, \rhd, \lhd, \circ)$ be an NS-pre-Lie algebra. Then $(A, \diamond, \smallblackdiamond)$ is an NS-Lie algebra, where
        \begin{align*}
            x \diamond y := x \rhd y - y \lhd x ~~~ \text{ and } ~~~ x\smallblackdiamond y := x \circ y- y \circ x, \text{ for } x, y \in A.
        \end{align*}
    \end{proposition}
    \begin{proof}
    For any $x,y,z \in A$, we have
        \begin{align*}
            x\diamond (y\diamond z) &= x \rhd (y \rhd z)- x \rhd (z \lhd y) - (y \rhd z)\lhd x + (z\lhd y)\lhd x,\\
            (x\diamond y)\diamond z &= (x \rhd y)\rhd z-(y \lhd x)\rhd z- z\lhd (x \rhd y)+ z\lhd (y \lhd x),\\
            (x\smallblackdiamond y)\diamond z &= (x\circ y)\rhd z-(y\circ x)\rhd z-z\lhd (x\circ y)+ z\lhd (y\circ x).
        \end{align*}
        Thus, we have
        \begin{align*}
            &x\diamond (y\diamond z)-(x\diamond y)\diamond z- y \diamond (x\diamond z)+ (y\diamond x)\diamond z-(x\smallblackdiamond y)\diamond z\\
            &= x \rhd (y \rhd z) - x \rhd (z \lhd y) - (y \rhd z)\lhd x + (z\lhd y)\lhd x - (x \rhd y)\rhd z + (y \lhd x)\rhd z + z\lhd (x \rhd y)\\& \quad - z\lhd (y \lhd x) - y \rhd (x \rhd z) + y \rhd (z \lhd x) + (x \rhd z)\lhd y - (z\lhd x)\lhd y + (y \rhd x)\rhd z - (x \lhd y)\rhd z\\& \quad - z\lhd (y \rhd x) + z\lhd (x \lhd x)-(x\circ y)\rhd z + (y\circ x)\rhd z + z\lhd (x\circ y) - z\lhd (y\circ x)\\
            &= \big( x \rhd (y \rhd z) - y \rhd (x \rhd z) -(x \rhd y+ x\lhd y + x \circ y)\rhd z + (y \rhd x + y \lhd x + y \circ x)\rhd z\big)\\
            & \quad + \big(z \lhd(x \rhd y+ x\lhd y + x \circ y)-(z\lhd x)\lhd y +(x \rhd z)\lhd y-x\rhd(z \lhd y)\big)\\
            & \quad + \big((z\lhd y)\lhd x+y\rhd(z\lhd x)-(y\rhd z)\lhd x-z \lhd (y \rhd x + y \lhd x + y \circ x)\big)\\
            &= 0.
        \end{align*}
        This proves the first identity of an NS-Lie algebra. Next, observe that
        \begin{align*}
            \dcb{x}{y}= x\diamond y-y\diamond x + x\smallblackdiamond y= x\odot y- y \odot x, 
        \end{align*}
        where $\odot$ is defined as in Definition \ref{nspl}. Note that the identity (\ref{nspl3}) can be written as
        \begin{align}\label{cyclic}
            -(x\odot y)\circ z ~+~ &x\circ (y\odot z)-(x\circ y)\lhd z+ x \rhd (y \circ z)+ (y\odot x)\circ z-y\circ (x\odot z) \\ & +(y\circ x)\lhd z- y \rhd (x \circ z)=0, \nonumber
            \end{align}
By considering the two more identities similar to (\ref{cyclic}) taking the variables $x, y, z$ in the cyclic order, and then adding all three identities, we get
                \begin{align*}
            &\big(x \circ ( y\odot z- z\odot y)-( y\odot z- z\odot y)\circ x\big) + \big(y \circ ( z\odot x- x\odot z)-( z\odot x- x\odot z)\circ y\big)\\&
            +\big(z \circ ( x\odot y- y\odot x)-( x\odot y- y\odot x)\circ z\big)+
            \big(x\rhd (y\circ z-z\circ y)-(y\circ z -z\circ y)\lhd x\big)\\
            &+\big(y\rhd (z\circ x-x\circ z)-(z\circ x -x\circ z)\lhd y\big)+
            \big(z\rhd (x\circ y-y\circ x)-(x\circ y -y\circ x)\lhd z\big)=0,
            \end{align*}
            equivalently, $x \smallblackdiamond \dcb{y}{z}+y\smallblackdiamond \dcb{z}{x}+z \smallblackdiamond \dcb{x}{y}+ x \diamond (y\smallblackdiamond z)+ y\diamond (z\smallblackdiamond x)+ z\diamond (x\smallblackdiamond y)=0.$
        Hence, the second identity of an NS-Lie algebra also holds. In other words, $(A, \diamond, \smallblackdiamond)$ is an NS-Lie algebra.
    \end{proof}

    Let $(A, \ast , \curlyvee)$ be an NS-commutative algebra. Then it can be realized as an NS-algebra and hence an NS-pre-Lie algebra $(A, \rhd_0, \lhd_0, \circ_0)$, where $ x \rhd_0 y = y \lhd_0 x = x \ast y$ and $x \circ_0 y = x \curlyvee y$, for all $x, y \in A.$
    
    \begin{thm}\label{thm-nsf-semi-classical}
        Let $(A, \ast , \curlyvee)$ be an NS-commutative algebra with the corresponding NS-pre-Lie algebra $(A, \rhd_0, \lhd_0, \circ_0)$. Let $(A[\![t]\!], \rhd_t, \lhd_t, \circ_t)$ be an NS-pre-Lie algebra deformation of $(A, \rhd_0, \lhd_0, \circ_0)$.
        Define two bilinear operations $\diamond, \smallblackdiamond : A \times A \rightarrow A$ by
        \begin{align*}
            x \diamond y :=~& \frac{x \rhd_t y - y \lhd_t x}{t}\Big|_{t=0} = x \rhd_1 y- y \lhd_1 x,\\
            x \smallblackdiamond y :=~& \frac{x \circ_t y - y \circ_t x}{t}\Big|_{t=0} = x \circ_1 y - y \circ_1 x,
        \end{align*}
        for $x, y \in A$. Then $(A, \ast , \curlyvee , \diamond , \smallblackdiamond)$ is an NS-$F$-manifold algebra.
    \end{thm}

    \begin{proof}
        Since $(A[\![t]\!], \rhd_t, \lhd_t, \circ_t)$ is an NS-pre-Lie algebra over the ring ${\bf k} [\![t]\!]$, it follows from Proposition \ref{prop-nspre-nsl} that $(A [\![t]\!], \diamond_t, \smallblackdiamond_t)$ is an NS-Lie algebra over ${\bf k} [\![t]\!]$, where
        \begin{align*}
            x \diamond_t y := x \rhd_t y - y \lhd_t x ~~~~ \text{ and } ~~~~ x \smallblackdiamond_t y := x \circ_t y - y \circ_t x, \text{ for } x, y \in A.
        \end{align*}
        In the NS-Lie algebra identities of $(A [\![t]\!], \diamond_t, \smallblackdiamond_t)$, if we compare the coefficient of $t^2$ in both sides of the identities, one obtains the NS-Lie algebra identities for $(A, \diamond, \smallblackdiamond)$ (similar to Theorem \ref{thm-defor}). In other words, $(A, \diamond, \smallblackdiamond)$ becomes an NS-Lie algebra.

        On the other hand, $(A[\![t]\!], \rhd_t, \lhd_t, \circ_t)$ is an NS-pre-Lie algebra implies that the identities (\ref{nspl1}), (\ref{nspl2}) and (\ref{nspl3}) hold if we replace $\rhd, \lhd, \circ$ by $\rhd_t, \lhd_t$ and $\circ_t$, respectively. In all these identities, we get a set of equations if we compare various powers of $t$. In particular, if we compare coefficients of $t$, we get three equations (each one comes from (\ref{nspl1}), (\ref{nspl2}), (\ref{nspl3}), respectively). Then, by a direct calculation (similar to the one given in Theorem \ref{thm-defor}), one can show that $(A, \ast, \curlyvee, \diamond, \smallblackdiamond)$ is an NS-$F$-manifold algebra.
    \end{proof}

    The NS-$F$-manifold algebra $(A, \ast, \curlyvee, \diamond, \smallblackdiamond)$ constructed in the above theorem is called the {\bf semi-classical limit} of the NS-pre-Lie algebra deformation $(A[\![t]\!], \rhd_t, \lhd_t, \circ_t)$ of the NS-commutative algebra $(A, \ast, \curlyvee).$ By taking $\curlyvee = 0$ in Theorem \ref{thm-nsf-semi-classical}, we get the following observation which is also a new result.

    \begin{remark}
        Let $(A, \ast)$ be a zinbiel algebra. Then it can be realized as an L-dendriform algebra $(A, \rhd_0 , \lhd_0)$ by setting $x \rhd_0 y = y \lhd_0 x = x \ast y$, for all $x, y \in A$. If $(A [\![t]\!], \rhd_t, \lhd_t )$ is an L-dendriform algebra deformation of $(A, \rhd_0, \lhd_0)$ then $(A, \ast, \diamond)$ is a pre-$F$-manifold algebra, where
                \begin{align*}
            x \diamond y :=~& \frac{x \rhd_t y - y \lhd_t x}{t}\Big|_{t=0} = x \rhd_1 y- y \lhd_1 x, \text{ for } x, y \in A.
            \end{align*}
            The pre-$F$-manifold algebra $(A, \ast, \diamond)$ is called the {\bf semi-classical limit} of the L-dendriform algebra deformation $(A [\![t]\!], \rhd_t, \lhd_t )$ of the zinbiel algebra $(A, \ast)$.
    \end{remark}

    


    \medskip

    \noindent {\bf Acknowledgements.} Anusuiya Baishya would like to acknowledge the financial support received from the Department of Science and Technology (DST), Government of India through INSPIRE Fellowship [IF210619]. Both authors thank the Department of Mathematics, IIT Kharagpur for providing the beautiful academic atmosphere where the research has been carried out.
    
\medskip

\noindent {\bf Conflict of interest statement.} On behalf of all authors, the corresponding author states that there is no conflict of interest.

\medskip

\noindent {\bf Data Availability Statement.} Data sharing does not apply to this article as no new data were created or analyzed in this study.

\end{document}